\documentclass[11pt,letter]{amsart}
\usepackage{amsfonts,amsmath,amssymb}
\usepackage{hyperref}
\usepackage{latexsym,fullpage,amsfonts,amssymb,amsmath,amscd,graphics,epic}
\usepackage[all]{xy}
\usepackage{amssymb,amsthm,amsxtra}
\usepackage[usenames]{color}
\usepackage{amscd}
\usepackage{amsthm}
\usepackage{amsfonts}
\usepackage{amssymb}
\usepackage{mathrsfs}
\usepackage{url}
\usepackage{bbm}
\usepackage{wasysym}
\usepackage{enumerate}
\usepackage{enumitem}

\theoremstyle{plain}% default
\newtheorem*{theorem*}{Theorem}
\newtheorem*{remark*}{Remark}
\newtheorem*{example*}{Example}
\newtheorem{lemma}{Lemma}
\newtheorem{proposition}[lemma]{Proposition}
\newtheorem{corollary}[lemma]{Corollary}
\newtheorem{theorem}[lemma]{Theorem}

\newtheorem*{conjecture*}{Conjecture}
\newtheorem{question}[lemma]{Question}

\newtheorem*{necessarycondition*}{Necessary Condition}

\numberwithin{lemma}{section}

\theoremstyle{definition}
\newtheorem{definition}[lemma]{Definition}

\newtheorem{example}[lemma]{Example}

\theoremstyle{remark}
\newtheorem{remark}[lemma]{Remark}

\newtheorem{notation}[lemma]{Notation}

\newcommand{\abs}[1]{\left|{#1}\right|}

\def\quotient#1#2{%
    \raise1ex\hbox{$#1$}\Big/\lower1ex\hbox{$#2$}%
}

\usepackage[normalem]{ulem}

\begin{document}

\date{October 16, 2022}
 \title[A property of ideals of jets of functions vanishing on a set]{A property of ideals of jets of functions vanishing on a set}
  \author{Charles Fefferman$^\dag$ and Ary Shaviv$^\dag$$^\dag$}
 \address{Department of Mathematics, Princeton University, Princeton, New
 Jersey 08544}
 \email{cf@math.princeton.edu; ashaviv@math.princeton.edu}
\thanks{$^\dag$C.F. was supported by the Air Force Office of Scientific Research (AFOSR),
under award FA9550-18-1-0069 and the National Science Foundation (NSF), under
grant DMS-1700180.}
\thanks{$^\dag$$^\dag$A.S.   was supported
by the Air Force Office of Scientific Research (AFOSR), under award FA9550-18-1-0069.
 }

\maketitle

\begin{abstract}
For a set $E\subset\mathbb{R}^n$ that contains the origin we consider $I^m(E)$ -- the set of all $m^{\text{th}}$ degree Taylor approximations (at the origin) of $C^m$ functions on $\mathbb{R}^n$ that vanish on $E$. This set is an ideal in $\mathcal{P}^m(\mathbb{R}^n)$ -- the ring of all $m^{\text{th}}$ degree Taylor approximations of $C^m$ functions on $\mathbb{R}^n$. Which ideals in  $\mathcal{P}^m(\mathbb{R}^n)$ arise as $I^m(E)$ for some $E$? In this paper we introduce the notion of a \textit{closed} ideal in  $\mathcal{P}^m(\mathbb{R}^n)$, and prove that any ideal of the form $I^m(E)$ is closed. We do not know whether in general any closed ideal is of the form $I^m(E)$ for some $E$, however we prove in \cite{FS2} that all closed ideals in $\mathcal{P}^m(\mathbb{R}^n)$ arise as $I^m(E)$ when $m+n\leq5$. 
\end{abstract}

\tableofcontents

\section{Introduction}

In this article we work in the ring $\mathcal{P}_{0}^{m}(\mathbb{R}^n)$ of $m$-jets (at the origin) of $C^m$ functions on $\mathbb{R}^n$ that vanish at the origin. We write $J^m(F)$ to denote the $m$-jet at $\vec 0$ of a function $F\in C^m(\mathbb{R}^n)$, i.e., its $m^{\text{th}}$ degree Taylor approximation about the origin. Given a subset $E\subset\mathbb{R}^n$ that contains the origin, we define an ideal $I^{m}(E)$ in $\mathcal{P}_{0}^{m}(\mathbb{R}^n)$ by $$I^{m}(E):=\{J^{m}(F)\text{ }|\text{ } F \in C^m(\mathbb{R}^n)\text{ and } F=0 \text{ on } E\}.$$ Our main goal is to understand which ideals $I$ in $\mathcal{P}_{0}^{m}(\mathbb{R}^n)$ arise as $I^m(E)$ for some $E$.

\

We introduce the notion that a given ideal $I$ \textit{implies} a particular jet $p$. The set of all the jets implied by a given ideal $I$ is an ideal containing $I$, which we call the \textit{closure} of $I$. We say that $I$ is \textit{closed} if it is equal to its closure.

\subsection*{Main results in a nutshell} We will prove that $I^{m}(E)$ is always closed (for any $m$ and any $E$); see Theorem \ref{main_theorem_on_necessary_condition}. We know of no example of an ideal $I$ that is closed but does not arise as $I^{m}(E)$ for some $E$. Indeed, in the rings $\mathcal{P}_{0}^{1}(\mathbb{R}^n),\mathcal{P}_{0}^{m}(\mathbb{R}^1),\mathcal{P}_{0}^{2}(\mathbb{R}^2),\mathcal{P}_{0}^{3}(\mathbb{R}^2)$ and $\mathcal{P}_{0}^{2}(\mathbb{R}^3)$ we will classify in our subsequent paper \cite{FS2} all possible closed ideals up to a natural equivalence relation, and for each ideal $I$ on these lists we will exhibit a set $E$ such that $I=I^m(E)$. The natural equivalence relation arises from the
fact that every $C^m$-diffeomorphism $\phi:\mathbb{R}^{n}\to\mathbb{R}^n$
that
fixes the origin induces an automorphism of $\mathcal{P}_0^m(\mathbb{R}^n)$
defined by $p\mapsto J^{m}(p\circ\phi)$. We say that two ideals $I$ and $I'$
in $\mathcal{P}_0^m(\mathbb{R}^n)$
are \textit{equivalent} if $I'$ is the image of $I$ under such an automorphism. Since this equivalence relation preserves the property of an ideal being closed, in \cite{FS2} we in particular prove that each closed ideal in the rings $\mathcal{P}_{0}^{1}(\mathbb{R}^n),\mathcal{P}_{0}^{m}(\mathbb{R}^1),\mathcal{P}_{0}^{2}(\mathbb{R}^2),\mathcal{P}_{0}^{3}(\mathbb{R}^2)$
and $\mathcal{P}_{0}^{2}(\mathbb{R}^3)$  arises as
$I^m(E)$
for some closed $E$ that contains the origin. 

\subsection*{Nullstellensatz-type point of view} Our problem is loosely analogous to the setting of the Hilbert Nullstellensatz, where one asks which polynomial ideals arise as $\mathcal{I}(V)$, the set of polynomials vanishing on an algebraic subset $V$ of $\mathbb{C}^n$. In that setting, we may say that an ideal $I$ \textit{implies} a polynomial $p$ if $p$ belongs to the radical of $I$. Recall that the Nullstellensatz tells us that an ideal $I\lhd\mathbb{C}[x_1,\dots, x_n]$ arises as $\mathcal{I}(V)$ for some algebraic subset  $V$ if and only if it is equal to its radical. When either $m=1$, $n=1$ or $m+n\leq5$, we will prove in  \cite{FS2} that an ideal
$I\lhd\mathcal{P}_0^m(\mathbb{R}^n)$ arises as $I^m(E)$ for some closed subset $E$ that contains the origin if and only
if it is equal to its closure. We have not yet studied enough examples to know whether to believe that for arbitrary $m$ and $n$ every
closed ideal in $\mathcal{P}_0^m(\mathbb{R}^n)$ arises as $I^{m}(E)$ for
some $E$. Perhaps, in addition to being closed, $I^{m}(E)$ has further properties
of which we are so far unaware.

We stress that unlike the setting of Hilbert Nullstellensatz, we are not aware of a natural map that assigns to each ideal $I\lhd\mathcal{P}_0^m(\mathbb{R}^n)$ a closed subset
$E$ that contains the origin such that $I=I^m(E)$. We present numerous
concrete examples in \cite{FS2}, and in particular it will become evident that it does not make sense to look at the zero locus of a given ideal -- we will see that in $\mathcal{P}_0^2(\mathbb{R}^2)$ (where $(x,y)$ is a standard coordinate system), $I^2(\{x=0\})$ is the principal ideal generated by $x$, while $I^2(\{x(x^2-y^3)=0\}\cap\{x\geq0\})$ is the principal ideal generated by $x^2$. The latter example also shows that $I^m(E)$ may be not radical (and in particular not real-radical).

\subsection*{A conjecture by N. Zobin} In \cite{FS2} we will see in particular that
each closed ideal in the rings  $\mathcal{P}_0^1(\mathbb{R}^n)$, $\mathcal{P}_0^m(\mathbb{R}^1),
 \mathcal{P}_0^2(\mathbb{R}^2),
\mathcal{P}_0^3(\mathbb{R}^2)$ and $\mathcal{P}_0^2(\mathbb{R}^3)$ arises
as $I^{m}(E)$ for a \textit{semi-algebraic} set\footnote{\label{semialgebraic_footnote}We
do not try to give the exact definition of semi-algebraic sets (and
maps) in this paper, but loosely speaking these may be thought of as sets
(and
maps) that can be described by finitely many polynomial equations and inequalities
and boolean operations. For detailed exposition on semi-algebraic geometry
see, for instance, \cite{BCR}.} $E$ that contains the origin. A conjecture
of N. Zobin (see \cite[Problem 5]{F}) asserts that every ideal $I^{m}(E)$ in $\mathcal{P}_0^m(\mathbb{R}^n)$
already arises as the ideal $I^{m}(V)$ for a semi-algebraic set $V$ that
contains the origin. Thus, we prove that the conjecture is true for the cases
in which either $m=1$, $n=1$ or $m+n\leq 5$. We do not know whether the conjecture
holds for general $m$ and $n$, though it seems plausible
to us.

\subsection*{Related problems} All the questions this paper addresses may
be asked with $C^m$ functions replaced by different classes of functions.
For instance, fix  $m,n\in\mathbb{N}$ and $m\leq k\in\mathbb{N}\cup\{\infty\}$.
Given a subset $E\subset\mathbb{R}^n$ that contains
the origin, define an ideal $I^{m}_k(E)$ in $\mathcal{P}_{0}^{m}(\mathbb{R}^n)$
by $$I^{m}_k(E):=\{J^{m}(F)\text{ }|\text{ } F \in C^k(\mathbb{R}^n)\text{
and } F=0 \text{ on } E\}.$$ What can we say about the ideal $I^m_k(E)$?
Which ideals in $\mathcal{P}^m_0(\mathbb{R}^n)$ arise as $I_{k}^{m}(E)$ for
some $E$? Does every ideal of the form $I^{m}_{k}(E)$ already arises as the
ideal $I^{m}_{k}(V)$ for a semi-algebraic set $V$? We would like to study
these variants, however in this paper we only deal with the case $m=k$.

\subsection*{The relation to extension problems}The ideal $I^{m}(E)$ arose naturally in connection with \textit{Whitney's extension problem}: given a function $f:E\to\mathbb{R}$ (with $\vec 0\in E\subset\mathbb{R}^n$) we want to decide whether $f$ extends from $E$ to a $C^m$ function $F$ on the whole $\mathbb{R}^n$. If such an extension $F$ exists, then $I^{m}(E)$ expresses the ambiguity in the jet of $F$ at the origin. Moreover, similarly to the definition of $I^m(E)$, we can define the ideal $I^m_{\vec
x}(E)$ of $m$-jets of $C^m$ functions that vanish identically on $E$ about
an arbitrary point $\vec x\in E$. In this terminology the ideal $I^m(E)$
is $I_{\vec 0}^m(E)$. Analyzing the ideals $I_{\vec
x}^m(E)$ is often a key step for solving such Whitney type extrapolation problems   (see, for instance, \cite{F1}). 

We stress here that for a given set $E$, one can in principle compute $I^{m}(E)$; see \cite{F1}. In this paper we are interested is the converse question: given $I$, can we find an $E$?

\subsection*{Intuition on implied jets} We prepare to motivate and explain the notion of an ideal $I$ implying a jet $p$ without going into technicalities. We sacrifice accuracy to meet this goal. Let us set up some ad-hoc notation and definitions for this purpose (these notation and definitions will be only used in this introduction; the main text will include a rigorous exposition and in particular a detailed construction of all the examples below). 

\

We identify $m$-jets with $m^{\text{th}}$ degree Taylor polynomials. Thus, $\mathcal{P}_0^m(\mathbb{R}^n)$ consists of all at most $m^{\text{th}}$  degree polynomials on $\mathbb{R}^n$ with a zero constant term, $J^m(F)$ is the $m^{\text{th}}$  degree Taylor polynomial of $F$ at the origin, and the multiplication in $\mathcal{P}_0^m(\mathbb{R}^n)$ is given by $PQ:=J^{m}(PQ)$.

\

Let $\Gamma\subset\mathbb{R}^n$ be an open set whose closure contains the origin. A $C^m$\textit{-flat} function on $\Gamma$ is a function $F\in C^m(\Gamma)$, such that $\partial^\alpha F(\vec x)=o(\abs{\vec x}^{m-\abs{\alpha}})$ as $\vec x$ tends to the origin in $\Gamma$, for any $\abs{\alpha}\leq m$. A $C^m$\textit{-tame} function on $\Gamma$ is a function $S\in C^m(\Gamma)$, such that $\partial^\alpha
S(\vec x)=O(\abs{\vec x}^{-\abs{\alpha}})$ as $\vec x$ tends to the origin
in $\Gamma$, for any $\abs{\alpha}\leq m$.

Often, $\Gamma$ will be a cone with vertex at the origin. If $\Gamma=\mathbb{R}^n\setminus \{ \vec 0\}$, then the $C^m$-flat functions are simply the $C^m$ functions $F$ with $J^{m}(F)=0$, restricted to $\mathbb{R}^n\setminus\{\vec 0\}$. The following simple observation will be important:

\begin{remark}\label{intro_label_0}Let $F$ be $C^m$-flat on $\Gamma$, and let $S$ be $C^m$-tame on $\Gamma$. Then, $S\cdot F$ is $C^m$-flat on $\Gamma$.\end{remark}

We start by presenting two examples to show that the ideal $I(E)$ unsurprisingly restricts the geometry of $E$.

\begin{example}\label{intro_example_1}In $\mathcal{P}_0^2(\mathbb{R}^2)$, suppose $x^{2}+y^{2} \in I^{2}(E)$. Then, by definition there exists $F\in C^2(\mathbb{R}^2)$ with jet $x^2+y^2$ that vanishes on $E$. We have $F(x,y)-(x^2+y^2)= o(x^2+y^2)$ as $(x,y)\to(0,0)$, hence $F(x,y)$ is nonzero on a punctured neighborhood of the origin. Consequently, that punctured neighborhood contains no points of $E$. This in turn easily implies that $I^{2}(E)$ consists of all jets that vanish at the origin, i.e., $I(E)=\mathcal{P}_0^2(\mathbb{R}^2)$. We thus proved that $$x^2+y^2\in
I^2(E)\implies (0,0)\text{ is an isolated point of }E\text{ and }I^2(E)=\mathcal{P}_0^2(\mathbb{R}^2).$$\end{example}

\begin{example}\label{intro_example_2}In $\mathcal{P}_0^2(\mathbb{R}^2)$, suppose $xy \in
I^{2}(E)$. Let $\Gamma$ be an open sector with vertex at the origin, whose
closure in $\mathbb{R}^2\setminus\{\vec 0\}$ meets neither the $x$-axis nor
the $y$-axis.  Then, by the argument of Example \ref{intro_example_1}, $E\cap
\Gamma$ cannot contain points arbitrarily close to the origin. This tells
us that the only possible "tangent directions to $E$ at the origin" are the
$x$-axis and the $y$-axis.\end{example}

We formalize the lesson of Examples \ref{intro_example_1} and \ref{intro_example_2} in the following definition:

\

Let $I$ be an ideal in $\mathcal{P}_0^m(\mathbb{R}^n)$ and let $\omega \in
S^{n-1}$ be a unit vector (a "direction"). We say that $\omega$ is \textit{forbidden}
for $I$ if there exist jets $Q_1,\dots,Q_L \in I$, a cone   $$\Gamma=\{\vec x\in
\mathbb{R}^n: 0<\abs{\vec x}<r, \abs{\frac{\vec x}{\abs{\vec x}}-\omega}<\delta\},$$
(for some $r,\delta>0$) and a constant $c>0$, such that $\abs{Q_1(\vec x)}+\dots \abs{Q_L(\vec x)}>c\abs{\vec
x}^m$ for any $\vec x\in\Gamma$. If $\omega$ is not forbidden, we say that $\omega$
is \textit{allowed}. We write $Forb(I)$ and $Allow(I)$, respectively, to
denote the sets of forbidden and allowed directions in the unit sphere.

\

For an ideal $I=I^m(E)$, the argument of Example \ref{intro_example_1} easily shows that only allowed directions may be tangent to $E$ at the origin. Thus, as promised, \textit{$I$ constrains the geometry of any set $E$ for which we hope that $I=I^m(E)$}.

\

Note that Example \ref{intro_example_1} also tells us that if an ideal is of the form $I^m(E)$ for some $E$ in the ring $\mathcal{P}_0^2(\mathbb{R}^2)$, and it contains the jet $x^2+y^2$, then it contains many more jets. Let us see more examples that illustrate how the existence of some jets in $I^m(E)$ can force the existence of another jet in $I^m(E)$.

\begin{example}\label{intro_example_3}In $\mathcal{P}_0^3(\mathbb{R}^2)$, suppose $x(x^2+y^2) \in I^{3}(E)$.
Then, $x^3, x^2y, xy^2 \in I^{3}(E)$ as well.

Indeed, since $x(x^2+y^2) \in I^{3}(E)$, there exists $F\in C^3(\mathbb{R}^2)$
with jet $x (x^2+y^2)$, such that $F=0$ on $E$. Equivalently, there
exists a $C^3$-flat function $F_1$ on $\mathbb{R}^2\setminus\{\vec 0\}$ such
that \begin{gather}\label{intro_label_1}x(x^2+y^2)+F_1(x,y)=0\text{
on }E.\end{gather}
Now set $S_1=\frac{x^2}{x^2+y^2}$, and note that $S_1$ is $C^3$-tame on $\mathbb{R}^2\setminus\{\vec
0\}$. Multiplying (\ref{intro_label_1}) by $S_1$, we have
\begin{gather}\label{intro_label_2}x^3+S_1\cdot F_1=0\text{ on }E.\end{gather}
Moreover, $S_1\cdot F_1$ is $C^3$-flat on $\mathbb{R}^2\setminus\{\vec 0\}$,
by Remark \ref{intro_label_0}. Therefore, (\ref{intro_label_2}) shows that
$x^3 \in I^{3}(E)$. Similar arguments using $S_2=\frac{xy}{x^2+y^2}$ and $S_3=\frac{y^2}{x^2+y^2}$
in place of $S_1$ show that also $x^2\cdot y$ and $xy^2$ belong to $I^{3}(E)$.\end{example}

The argument of Example \ref{intro_example_1} can be modified easily to show that any \textit{proper} closed ideal in $\mathcal{P}_0^m(\mathbb{R}^n)$ has a non-empty set of allowed directions. Example \ref{intro_example_3} shows that \textit{not} every ideal in $\mathcal{P}_0^m(\mathbb{R}^n)$ with allowed directions is closed: the principal ideal generated by $x(x^2+y^2)$ in $\mathcal{P}_0^3(\mathbb{R}^2)$ is not closed, though $\{(0,\pm1)\}$ are allowed directions of this ideal.

\

Our final example in this introduction brings into play the set of allowed directions:

\begin{example}\label{intro_example_4}  In $\mathcal{P}_0^2(\mathbb{R}^3)$, suppose $I=I^{2}(E)$ contains
$x^2$ and $y^2-xz$. Then, $I$ contains $xy$.

To see this, first note that outside any conic neighborhood of the $z$-axis, we have $\abs{x^2}+\abs{y^2-xz}>c(x^2+y^2+z^2)$ for some constant $c>0$ that depends on the conic neighborhood. 
Therefore, $Allow(I)$
is contained in $\{(0,0,\pm1)\}$ and so $E$ is tangent to the $z$-axis, i.e.,
\begin{gather}\label{stubbern_example_label_1}E\subset \{(x,y,z):\abs{(x,y)}\leq g(\abs{z})\cdot\abs{z}\},\end{gather}for some function $g:[0,\infty)\to\mathbb{R}$ that is strictly positive away from 0 and that satisfies \begin{gather}\label{stubbern_example_label_2}g(t)\to0\text{ as }t\to0.\end{gather}Let us see how, by possibly replacing $g$ by a function that goes to zero more slowly, we may assume without loss of generality that $g\in C^2(0,\infty)$ and \begin{gather}\label{stubbern_example_label_3}(\frac{d}{dt})^kg(t)=O(t^{-k}g(t))\text{
for all }k\in\{0,1,2\}.\end{gather}Note that replacing $E$ with its intersection with a small ball about the origin does not affect $I^2(E)$, so we may first replace $g$ by $\min\{g(t),1\}$. We thus may assume without loss of generality that $0<g(t)\leq1$ for all $t\in(0,\infty)$.
Second, we set $\tilde g(t):=\sup\{\frac{2t}{t+s}g(s):s\in(0,\infty)\}$, for all $t\in(0,\infty)$. Note that this sup is finite since $g(s)\leq1$ for all $s\in(0,\infty)$. By taking $s=t$ in the above sup, we see that \begin{gather}\label{stubbern_example_label_4}\tilde g(t)\geq g(t)\text{
for all }t\in(0,\infty).\end{gather} Fix $\epsilon>0$. By (\ref{stubbern_example_label_2}), there exists $\delta(\epsilon)>0$ such that $g(s)<\epsilon$ for all $0<s<\delta(\epsilon)$. So we have $(\frac{2t}{t+s})g(s)<2\epsilon$ for all $0<s<\delta(\epsilon)$. For $s\geq\delta(\epsilon)$ we have $(\frac{2t}{t+s})g(s)\leq\frac{2t}{t+s}\leq\frac{2t}{\delta(\epsilon)}<2\epsilon$ for $t<\tilde\delta(\epsilon):=\epsilon\cdot\delta(\epsilon)$. These two facts together tell us that $\tilde g(t)\leq2\epsilon$ for all $t<\tilde\delta(\epsilon)$, and so \begin{gather}\label{stubbern_example_label_5}\tilde
g(t)\to0\text{
as }t\to0.\end{gather}One readily sees that if $\frac{t_1}{t_2}\in[\frac{1}{2},2]$ then $\frac{1}{4}\cdot\frac{2t_1}{t_1+s}\leq \frac{2t_2}{t_2+s}\leq 4\cdot \frac{2t_1}{t_1+s}$ for all $s\in(0,\infty)$, and so \begin{gather}\label{stubbern_example_label_6}\frac{1}{4}\cdot\tilde
g(t_1)\leq \tilde g(t_2)\leq 4\cdot\tilde g(t_1)\text{
if }\frac{t_1}{t_2}\in[\frac{1}{2},2].\end{gather} 
Let $\varphi\in C^\infty((0,\infty))$ be a non-negative function, supported in $[\frac{1}{2},2]$, and not the zero function, and set $$g^*(t):=\int\limits_{0}^{\infty}\varphi(\frac{t}{s})\cdot\tilde g(s)\frac{ds}{s}.$$ Then $g^*\in C^\infty((0,\infty))$. Since $\frac{t}{s}\in[\frac{1}{2},2]$ in the support of the integrand, we have from (\ref{stubbern_example_label_6}) that \begin{gather}\label{stubbern_example_label_7}c\cdot\tilde
g(t)\leq g^*(t)\leq C\cdot\tilde g(t)\text{
for all }t\in(0,\infty),\end{gather} where $c,C$ above are positive constants depending only on $\varphi$. Moreover, for any $t\in(0,\infty)$ and $k\in\{0,1,2\}$ we have \begin{gather}\label{stubbern_example_label_9}\abs{(\frac{d}{dt})^kg^*(t)}=\abs{\int\limits_0^{\infty}s^{-k}\varphi^{(k)}(\frac{t}{s})\cdot\tilde g(s)\frac{ds}{s}}\leq \int\limits_0^\infty s^{-k}\abs{\varphi^{(k)}(\frac{t}{s})}\frac{ds}{s}\cdot4\tilde
g(t)=C' t^{-k}\tilde g(t),\end{gather}where $C'$ is a positive constant
depending only on $\varphi$, and again we exploit estimate (\ref{stubbern_example_label_6}) in the support of the integrand. Now, (\ref{stubbern_example_label_7}) and (\ref{stubbern_example_label_9})  imply that \begin{gather}\label{stubbern_example_label_10}(\frac{d}{dt})^kg^{*}(t)=O(t^{-k}g^{*}(t))\text{
for all }k\in\{0,1,2\}.\end{gather} From (\ref{stubbern_example_label_4}), (\ref{stubbern_example_label_5}) and (\ref{stubbern_example_label_7}) we see that \begin{gather}\label{stubbern_example_label_11}g^*(t)\to0\text{
as }t\to0\end{gather} and \begin{gather}\label{stubbern_example_label_12}C''
g^*(t)\geq g(t)\text{
for all }t\in(0,\infty),\end{gather}where $C''$ is a positive constant
depending only on $\varphi$. Finally, thanks to (\ref{stubbern_example_label_10}), (\ref{stubbern_example_label_11}) and (\ref{stubbern_example_label_12}) we may replace $g^*(t)$ by $g^+(t):=C''g^*(t)$ and conclude that (\ref{stubbern_example_label_1}), (\ref{stubbern_example_label_2}) and (\ref{stubbern_example_label_3}) all hold.

We thus established that we may assume without loss of generality that (\ref{stubbern_example_label_1}),
(\ref{stubbern_example_label_2}) and (\ref{stubbern_example_label_3}) all
hold. We proceed by replacing $E$ with its intersection
with a small ball about the origin (this does not affect $I^2(E)$), and so we may also assume that $\abs{(x,y)}\leq\abs{z}$ for any $(x,y,z)\in E$.

Let $\theta(t)$ be a $C^2$ (cutoff) function on $[0,\infty)$, supported in $[0,2]$ and equal to 1 on $[0,1]$. We define functions on $\mathbb{R}^3\setminus\{\vec0\}$ by $$F_1(x,y,z)=\frac{y^3}{z}\cdot\theta(\frac{\abs{(x,y)}}{g(z)\abs{z}})\text{ and }S(x,y,z)=\frac{y}{z}\cdot\theta(\frac{\abs{(x,y)}}{\abs{z}}).$$ One can readily verify that (\ref{stubbern_example_label_2}) and (\ref{stubbern_example_label_3}) imply that $F_1$ is $C^2$-flat on $\mathbb{R}^3\setminus\{\vec 0\}$ and $S$ is $C^2$-tame on $\mathbb{R}^3\setminus\{\vec 0\}$. Recall that $y^2-xz\in I^2(E)$. Thus, there exists a $C^2$-flat function on $\mathbb{R}^3\setminus\{\vec 0\}$ which we denote by $F_2$, such that \begin{gather}\label{stubbern_example_label_13}y^2-xz+F_2(x,y,z)=0\text{
on }E\setminus\{\vec 0\}.\end{gather}Multiplying (\ref{stubbern_example_label_13}) by $-\frac{y}{z}$ we find that \begin{gather}\label{stubbern_example_label_14}xy-\frac{y^3}{z}-\frac{y}{z}F_2(x,y,z)=0\text{
on }E\setminus\{\vec 0\}.\end{gather}By our assumption that that $\abs{(x,y)}\leq\abs{z}$ for any $(x,y,z)\in E$ and (\ref{stubbern_example_label_1}) we have that $$\theta(\frac{\abs{(x,y)}}{g(z)\abs{z}})=\theta(\frac{\abs{(x,y)}}{\abs{z}})=1\text{ on }E,$$and so $$F_1=\frac{y^3}{z}\text{ and }S=\frac{y}{z}\text{ on }E.$$Consequently, (\ref{stubbern_example_label_14}) implies that \begin{gather}\label{stubbern_example_label_15}xy-[F_1+SF_2]=0\text{
on }E\setminus\{\vec 0\}.\end{gather} Recall that $F_1,F_2$ are $C^2$-flat on $\mathbb{R}^3\setminus\{\vec0\}$ and $S$ is $C^2$-tame
on $\mathbb{R}^3\setminus\{\vec0\}$. Thanks to Remark
\ref{intro_label_0} we see that $F_2+SF_2$ is $C^2$-tame
on $\mathbb{R}^3\setminus\{\vec0\}$, and so (\ref{stubbern_example_label_15}) tells us that $xy\in I^2(E)$. This completes
our analysis of Example \ref{intro_example_4}. 

We stress that we got crucial help from our knowledge of $Allow(I)$, which told us where $E$ lives and permitted us to make use of the cutoff functions $\theta(\frac{\abs{(x,y)}}{g(z)\abs{z}})$ and $\theta(\frac{\abs{(x,y)}}{\abs{z}})$, thus avoiding the singularities of the functions $\frac{y^{3}}{z}$ and $\frac{y}{z}$.\end{example}

\

We are now ready to define the notion of an implied jet. Unfortunately, it is not so simple.

\

Let $I$ be an ideal in $\mathcal{P}_0^m(\mathbb{R}^n)$, and let $\Omega\subset S^{n-1}$ be the set of allowed directions for $I$. We say that $I$ \textit{implies} a given jet $p$ if there exist a constant $A>0$ and jets $Q_1,\dots Q_L\in I$ for which the following holds:

Given $\epsilon>0$ there exist $\delta, r >0$ such that for any $0<\rho<r$ there exist functions $F, S_1,\dots ,S_L$ satisfying 

\begin{gather}\label{intro_label_6}\abs{\partial^\alpha
F(\vec x)}\leq \epsilon \rho^{m-\abs{\alpha}} \text { for all }\frac{\rho}{4}<\abs{\vec x}<4\rho
\text{ and all }\abs{\alpha}\leq m; \end{gather}
\begin{gather}\label{intro_label_7}\abs{\partial^\alpha S_l(\vec
x)}\leq A \rho^{-\abs{\alpha}} \text { for all }\frac{\rho}{4}<\abs{\vec
x}<4\rho
\text{, all }\abs{\alpha}\leq m \text{ and all }1\leq l\leq L; \end{gather}
\begin{multline}\label{intro_label_8}\text{if }\Omega\neq\emptyset\text{
then }p(\vec x)=F(\vec x)+S_1(\vec x)Q_1(\vec x)+S_2(\vec
x)Q_2(\vec x)+\dots+S_L(\vec
x)Q_L(\vec x) \\ \text {for all }\frac{\rho}{2}<\abs{\vec
x}<2\rho
\text{ such that }\text{dist}(\frac{\vec x}{\abs{\vec x}},\Omega)<\delta.
\end{multline}
The point of this definition is as follows: suppose that $I$ implies $p$, and suppose that $I=I^{m}(E)$ for some $E$. The functions $S_1,\dots, S_L, F$ in (\ref{intro_label_6})-(\ref{intro_label_8}) are allowed to depend on the length scale $\rho$, but by patching together the $S_1,\dots, S_L, F$ arising from all small length scales, we can find functions  $S^\#_1,\dots S^\#_L$ that are $C^m$-tame on $\mathbb{R}^n\setminus\{\vec 0\}$ 
and a function $F^\#$ that is $C^m$-flat on $\mathbb{R}^n\setminus\{\vec 0\}$ such that in some punctured neighborhood of the origin we have $$p(\vec x)=F^{\#}(\vec x)+S^{\#}_1(\vec x)Q_1(\vec x)+S^{\#}_2(\vec
x)Q_2(\vec x)+\dots+S^{\#}_L(\vec
x)Q_L(\vec x)\text{ on }E.$$ An easy argument using Remark \ref{intro_label_0} then shows that $p \in I=I^{m}(E)$. Thus, if $I=I^{m}(E)$ then $I$ contains any jet $p$ implied by $I$, i.e., $I^{m}(E)$ is always closed (Theorem \ref{main_theorem_on_necessary_condition} below).

\

Let us see how the above definition applies to Example \ref{intro_example_4} by verifying that $I$ implies $p=xy$. Recall that $\Omega=Allow(I)$ is contained in $\{(0,0,\pm1)\}$. We take $Q_1$ to be $y^2-xz$, and we take (for instance) $A=10^9$. Given $\epsilon>0$, we take (for instance) $\delta=r=\frac{\epsilon}{10^9}$.

Now, given $\rho<r$, we must produce functions $F$ and $S_1$ satisfying (\ref{intro_label_6}),(\ref{intro_label_7}), and (\ref{intro_label_8}) above.
We will take $$F=\frac{y^3}{z}\cdot \theta(\frac{\abs{(x,y)}}{\delta\cdot\rho})$$and$$S_1=-\frac{y}{z}\cdot\theta(\frac{\abs{(x,y)}}{\rho}),$$
where $\theta$ is a $C^3$-smooth one variable cutoff function defined on $[0,\infty]$, equal to 1 on $[0,4]$, supported on $[0,8]$ and such that $\theta$ and its derivatives up to order 3 have absolute value at most 100.

The above $F$ and $S_1$ satisfy (\ref{intro_label_6}),(\ref{intro_label_7}),
and (\ref{intro_label_8}). In fact, for $\frac{\rho}{4}<\abs{\vec x}<4\rho$, (\ref{intro_label_7}) is easily  verified, and (\ref{intro_label_6}) holds since $\abs{\partial^\alpha
F(\vec x)}=O(\delta^{3-\abs{\alpha}}\cdot\rho^{2-\abs{\alpha}})$ for all $\abs{\alpha}\leq 2$. To check (\ref{intro_label_8}) we note that $\theta(\frac{\abs{(x,y)}}{\delta\rho})=\theta(\frac{\abs{(x,y)}}{\rho})=1$ for $\frac{\rho}{2}<\abs{(x,y,z)}<2\rho$ such that $\text{dist}(\frac{\vec x}{\abs{\vec x}},\Omega)<\delta$. Consequently, (\ref{intro_label_8}) reduces to the equation $xy=(-\frac{y}{z})(y^2-xz)+\frac{y^3}{z}$. Thus, as promised, $I$ implies $xy$ in Example \ref{intro_example_4}. The cutoff function $\theta(\frac{\abs{(x,y)}}{\delta\cdot\rho})$
here plays the role of the cutoff $\theta(\frac{\abs{(x,y)}}{g(\abs{z})\cdot\abs{z}})$
in Example \ref{intro_example_4}.

\subsection*{Calculating implied jets} So we have defined the notion that
an ideal $I$ implies a jet
$p$. How can we show in practice that a \textit{given} ideal $I$ implies a \textit{given} jet
$p$?
In Section \ref{section_calculation_method} we develop tools that answer
this question,
by  introducing the notions of \textit{strong implication} and \textit{strong directional
implication}. These notions use the definition of \textit{negligible functions}, that
is quite technical, but relatively easy to work with. Then, we show that
if a given jet $p$ is strongly implied by a given ideal $I$ in any allowed direction
of $I$, then the jet $p$ is implied by $I$ (Corollary \ref{cor_strong_directional_implication_imply_implication}).
 We also provide an easy algorithm to calculate the set of allowed directions
of a given ideal (Corollary \ref{old_new_cor_on_how_to_calc_allow_with_inclusion}) -- if we are given a basis of an ideal $I$, this algorithm always produces a set that contains $Allow(I)$, and sometimes it calculates $Allow(I)$ exactly.
We routinely use these tools in \cite{FS2} when we exhibit many examples
of closed ideals.

\subsection*{Calculating the closure of an ideal} Recall
once more that the closure of $I$ consists of all the jets implied by $I$,
and that $I$ is closed if and only if it is equal to its closure. How can we calculate the closure of a given ideal $I$, i.e., how can
we calculate a basis for the space of all implied jets of a given ideal?
We answer this question in Section \ref{section_algorithm}. In this
section we show how to realize the conditions defining implied jets as conditions
about the existence of sections of some semi-algebraic bundles (bundles in
the sense of Fefferman-Luli; see \cite{FL1}). Consequently, we also explain how, in principle, we
can calculate the closure of a given ideal using well known results regarding
the spaces of sections of such bundles. This algorithm is based on results from \cite{FL2,FL3}.

\

We stress that the tools in Section \ref{section_calculation_method}  provide useful information in many particular cases in which we want to show that a given ideal implies a given jet, but (as far as we know) are not guaranteed to work in general. On the other hand, the algorithm in Section \ref{section_algorithm} is guaranteed to compute the closure of any given ideal of $m$-jets, but is unfortunately far too labor-intensive to use in practice.

\subsection*{Acknowledgments}We are grateful to Nahum Zobin for asking the beautiful Question \ref{semi_algebraic_question}. We thank the participants of the Whitney Extension Problems Workshops over the years, the institutions that hosted and supported these workshops -- American Institute if Mathematics, Banff International Research Station, the College of William and Marry, the Fields Institute, the Technion, Trinity College Dublin and the  University of Texas at Austin -- and the NSF, ONR, AFOSR and BSF for generously supporting these workshops. 

\section{Closed ideals and a necessary condition}

As said, we start a rigorous construction of our theory below, and do not rely on anything defined in the introduction. Let us start by fixing notation.

\subsection{Notation}We work in
$\mathbb{R}^n$ with Euclidean
metric, and most of the notation we use is standard. 

\

\textbf{Functions.} For an open subset $U\subset\mathbb{R}^n$ and $m\in\mathbb{N}\cup\{0\}$
we denote by $C^m(U)$ the space of real valued $m$-times continuously differentiable
functions. 
We use multi-index notation for derivatives: for a multi-index $\alpha:=(\alpha_1,\alpha_2,\dots,\alpha_n)\in(\mathbb{N}\cup\{0\})^n$
we set $\abs{\alpha}:=\alpha_1+\dots+\alpha_n$ and $\alpha!:=\alpha_1!\alpha_2!\dots\alpha_n!$.
For $\vec x=(x_1,x_2,\dots, x_n)\in\mathbb{R}^n$ we set $\vec x^{\alpha}:=x_1^{\alpha_1}x_2^{\alpha_2}\dots
x_n^{\alpha_n}$. If $\abs{\alpha}\leq m$ and $f\in C^m(U)$ we write $$f^{(\alpha)}:=\frac{\partial^{|\alpha|}f}{\partial
x_1^{\alpha_1}\partial x_2^{\alpha_2}\cdots\partial
x_n^{\alpha_n}}$$ when $\abs{\alpha}\neq0$, and $f^{(\alpha)}:=f$ when $\abs{\alpha}=0$.
We sometimes write $\partial^\alpha f$ or $\partial_\alpha f$ instead of
$f^{(\alpha)}$. When it is clear from the context we will sometimes write $f_{xy}$
or $\partial^2_{xy}f$ when $\alpha=(1,1,0,0,\dots,0)$ and $(x,y,z,\dots)$
is a coordinate system on $\mathbb{R}^n$, and other such similar conventional
notation. 

\

\textbf{Asymptotic behaviors.} For $f,g\in C^m(U)$ and $\vec x_0\in U$ we
write $$f(\vec x)=o(g(\vec x))\text{ as }\vec x\to\vec x_0$$ if$$\frac{f(\vec
x)}{g(\vec x)}\to0\text{ as }\vec x\to\vec x_0.$$ If $\vec x_0$ is not specified
then $\vec x_0=\vec 0$, unless otherwise is clear from the context. We write
$$f(\vec x)=O(g(\vec
x))\text{ as }\vec x\to\vec x_0$$ if$$\abs{\frac{f(\vec x)}{g(\vec x)}}\text{
is bounded in some punctured neighborhood of }\vec x_0.$$ Again, if $\vec
x_0$ is not specified then $\vec x_0=\vec
0$, unless otherwise is clear from the context. 

\

\textbf{Balls, cones, annuli and other geometric objects.}  For $r>0$ we
set $B(r):=\{\vec
x\in\mathbb{R}^n:\abs{\vec
x}<r\}$ and set $B^{\times}(r):=B(r)\setminus\{\vec 0\}$, where $n$ should
be
clear from the context. For two sets $X,Y\subset\mathbb{R}^n$ we
set $$\text{dist}(X,Y):=\inf\{\abs{\vec x-\vec y}:\vec x\in X,\vec y\in Y\}\text{ if }X\neq\emptyset\text{ and }Y\neq\emptyset\text{ };\text{ }\text{dist}(X,Y)=+\infty\text{ otherwise,}$$
and if one of them is a singleton we write $\text{dist}(\vec x,Y)$ instead
of $\text{dist}(\{\vec x\},Y)$, and similarly if the other is a singleton. We denote
as usual $\mathbb{R}^n\supset S^{n-1}:=\{\abs{\vec x}=1\}$ and refer to points in
$S^{n-1}$ as \textit{directions}. For $\Omega\subset S^{n-1}$
and $\delta>0$ we denote (the
dome around $\Omega$ of opening $\delta$) $$D(\Omega,\delta):=\{\vec
\omega\in S^{n-1}: \text{dist}(\omega,\Omega)<\delta\}.$$Note that in particular if $\Omega=\emptyset$ then $D(\Omega,\delta)=\emptyset$. Given (radius)
$r>0$  we set $$\Gamma(\Omega,\delta,r):=\bigcup\limits_{r'\in(0,r)}r' D(\Omega,\delta)=\{\vec
x\in\mathbb{R}^n:0<\abs{\vec x}<r,\text{dist}(\frac{\vec x}{\abs{\vec x}},\Omega)<\delta\}.$$Note that in particular
if $\Omega=\emptyset$ then $\Gamma(\Omega,\delta,r)=\emptyset$. 

For
a singleton $\omega\in S^{n-1}$ we write $D(\omega,\delta)$ and $\Gamma(\omega,\delta,r)$
instead of  $D(\{\omega\},\delta)$ and $\Gamma(\{\omega\},\delta,r)$ (respectively).
We call a set of the form $\Gamma(\omega,\delta,r)$ (resp. $\Gamma(\Omega,\delta,r)$)
a cone in the direction $\omega$ (around the set of directions $\Omega\subset
S^{n-1}$), or a conic neighborhood of $\omega$ (of $\Omega$). Note that some
cones are non-convex. Also note that for any $\Omega\subset
S^{n-1}$, $\delta>0$ and $r>0$, we have that $\Gamma(\Omega,\delta,\epsilon)$
is open in $\mathbb{R}^n$, and $D(\Omega,\delta)$ is open in $S^{n-1}$ (in the
restricted topology). Finally (when $n$ is clear from the context), for $\mathbb{R}\ni
K\geq 1$ and $\mathbb{R}\ni r>0$ we define the annulus$$\text{Ann}_K(r):=\{\vec
x\in\mathbb{R}^n:\frac{r}{K}<\abs{\vec x}<Kr\}.$$

\subsection{Basic Definitions}

\begin{definition}[jet spaces] Let $m,n\in\mathbb{N}$. For a real
valued function $f\in C^m(\mathbb{R}^n)$ we define the \textit{$m^{\text{th}}$ degree
jet of $f$ at the origin} (or $m$-jet), denoted by $J^m(f)$, to be the $m^{\text{th}}$
degree Taylor polynomial (that has degree at most $m$) of $f$ at the origin.
The jet $J^m(f)$ is an element
of the space of (at most) $m^{\text{th}}$ degree (Taylor) polynomials in
$n$ variables, which we call the $m^{\text{th}}$ degree jet
space at the origin, and denote by $\mathcal{P}^m(\mathbb{R}^n)$. This is naturally
a (commutative unital)
ring, with multiplication given by $PQ:=J^{m}(PQ)$, however it is not a integral domain, e.g., in $\mathcal{P}^m(\mathbb{R}^1)$
we always have  $x^{m}\cdot x=0$. We stress that by $PQ$ (and $P\cdot Q$) we always mean jet product, and not the produce in the ring of polynomials, unless we say otherwise.
\end{definition}
Note that  $\mathcal{P}^m(\mathbb{R}^n)$ is a finite
dimensional
vector space, and that for
any $f,g\in C^m(\mathbb{R}^n)$ we have $J^m(f\cdot g)=J^m(f)\cdot J^m(g)$. In order
to avoid confusion we will use the following notation:

\begin{notation}
Let $\mathcal{F}\subset \mathcal{P}^m(\mathbb{R}^n)$ be a family of jets. We denote by $\langle \mathcal{F} \rangle_m$
the ideal in $\mathcal{P}^m(\mathbb{R}^n)$
generated by $\mathcal{F}$. For instance, in the single variable case we
have $\langle x \rangle_m=\text{span}_{\mathbb{R}}\{x,x^2,\dots,x^m\}$.
\end{notation}

\begin{definition}[order of vanishing of functions]\label{order-of-vanishing-def}Let
$m,n\in\mathbb{N}$, and let $f\in C^m(\mathbb{R}^n)$. \textit{The order of vanishing
(at the origin)} of $f$ is said to be

\begin{itemize}
\item the minimal $1\leq m'\leq m$ such that $J^{m'}(f)\neq0$, if
such $m'$ exists and $f(\vec 0)=0$;
\item more than $m$, if $J^m(f)=0$;
\item 0, if $f(\vec 0)\neq 0$. 
\end{itemize}
\textit{The order of vanishing of a jet} $p\in \mathcal{P}^m(\mathbb{R}^n)$
is the order of vanishing of the polynomial $p$. 
  
\end{definition} 

Note that every $C^m$-diffeomorphism $\phi:\mathbb{R}^n\to\mathbb{R}^n$ that
fixes the origin induces an automorphism of $\mathcal{P}^m(\mathbb{R}^n)$
defined by $p\mapsto J^{m}(p\circ\phi)$. This automorphism does \textit{not} preserve in general the degree (as a polynomial) of a jet, however it always preserves the order of vanishing.

\begin{definition}
We set $$\mathcal{P}_0^m(\mathbb{R}^n):= \{p\in \mathcal{P}^m(\mathbb{R}^n)|\text{
the order
of vanishing of }p\text{ is not }0\}.$$  
\end{definition} Note that $\mathcal{P}_0^m(\mathbb R ^n)$ is a (commutative) ring, but
it is not unital. It is also a finite dimensional vector space. Equivalently,
one can define $\mathcal{P}_0^m(\mathbb{R}^n)$ as the unique maximal ideal
in $\mathcal{P}^m(\mathbb{R}^n)$.

\begin{definition}\label{defintion_of_I_m_of_E}Let $E\subset \mathbb{R}^n$ be a
closed set containing the
origin. We define $$I^m(E):=\{p\in \mathcal{P}^m(\mathbb{R}^n)
| \exists f\in C^m(\mathbb{R}^n), f|_E=0,J^m(f)=p\}.$$
\end{definition}

\subsubsection{Simple observations}\label{simple_observations} Let $E,E'\subset
\mathbb{R}^n$ be two closed sets, both containing the origin. 
\begin{enumerate}[label=(\roman*)]
\item $I^m(E)$ only depends on the local behavior of $E$ around the origin:
if there exists $r>0$ such that $E\cap B(r)=E'\cap B(r)$,
then $I^m(E)=I^m(E')$. 

\item $I^m(E)$ is an ideal in $\mathcal{P}^m(\mathbb{R}^n)$: indeed, let $p\in I^m(E)$
and
$p'\in \mathcal{P}^m(\mathbb{R}^n)$. Then, there exists $f\in C^m(\mathbb{R}^n)$ such that
$J^m(f)=p$
and $f|_E=0$. Note that $f\cdot p'$ is a $C^m$ function on $\mathbb{R}^n$
that vanishes on $E$, and $J^{m}(f\cdot p')=p\cdot p'$. We conclude
that $p\cdot p'\in I^m(E)$, i.e. $I^m(E)\lhd \mathcal{P}^m(\mathbb{R}^n)$.

\item $I^m(E)\subset \mathcal{P}_0^m(\mathbb{R}^n)$, in fact it is an ideal:
$I^m(E)\lhd \mathcal{P}_0^{m}(\mathbb{R}^n)$.

\item $I^m(E\cup E')\subset I^m(E)\cap I^m(E')$. 

\item $I^{m}(E)\subset I^m(E\cap E')$. 

\end{enumerate}

\subsection{Allowed and Forbidden directions of an ideal and tangent directions of a set}
\begin{definition}[allowed and forbidden directions of jets and ideals]\label{definition_allowed_and_forbidden_directions}
\
\begin{itemize}
\item Let $p_{1},p_2,\dots,p_L\in\mathcal{P}_0^m(\mathbb{R}^n)$ be jets and let $\omega\in S^{n-1}$ be a direction. We say that
$\omega$ is a \textit{forbidden direction} of $p_1,\dots p_L$ if the following
holds:
\begin{multline}\label{equiv_cond_one_for_forbidden}
  \text{There exist }c,\delta,r>0 \text{
such that} \\ \abs{p_1(\vec x)}+\abs{p_2(\vec x)}+\dots+\abs{p_L(\vec x)}> c\cdot\abs{\vec
{x}}^m \text { for all }\vec x\in\Gamma(\omega,\delta,r).
\end{multline}Otherwise, we say that $\omega$ is an \textit{allowed direction}
of $p_1,\dots p_L$.
\item Let $I\lhd\mathcal{P}_0^{m}(\mathbb{R}^n)$ be an ideal. A direction $\omega\in
S^{n-1}$
is said to be a \textit{forbidden direction} of $I$, if there exist $p_{1},p_2,\dots,p_L\in
I$ such that $\omega$ is a forbidden direction of $p_{1},p_2,\dots,p_L$. Otherwise, we say that
$\omega$ is an \textit{allowed direction} of $I$.
\end{itemize}
We denote
the sets of forbidden and allowed directions of $I$ by $Forb(I)$ and $Allow(I)$
respectively. Note that the set $Allow(I)\subset S^{n-1}$ is always closed.
\end{definition}

\begin{definition}[tangent, forbidden and allowed directions of a set]\label{def_tangent_directions_of_a_set}Let
$E\subset\mathbb{R}^n$ be a closed subset containing the origin
and let $\omega\in S^{n-1}$. We say that \emph{$E$ is tangent to the direction
$\omega$} if \begin{gather}E\cap\Gamma(\omega,\delta,r)\neq\emptyset\text{
for all }\delta>0 \text{ and all }r>0.\end{gather} We denote by $T(E)\subset
S^{n-1}$ the set of all directions to which $E$
is tangent. Finally, for a fixed $m\in\mathbb{N}$, we say that \emph{$\omega$
is a forbidden (resp. allowed) direction
of $E$}, if $\omega$ is a forbidden (allowed) direction of $I^m(E)$. We denote
the sets of forbidden and allowed directions of $E$ by $Forb(E)$ and $Allow(E)$
respectively (here $m$ should be implicitly
understood from the context).\end{definition}

\begin{lemma}\label{tangent_subset_bad_lemma}Let $\vec 0\in E\subset\mathbb{R}^n$
be a closed subset and $m\in\mathbb{N}$. Then, $T(E)\subset Allow(E)$.\end{lemma}
\begin{proof}Let $\omega\in T(E)$. Assume towards a contradiction that $\omega\in
Forb(E)=Forb(I^m(E))$. By Definition \ref{definition_allowed_and_forbidden_directions}
and (\ref{equiv_cond_one_for_forbidden}) there exists $p_{1},\dots,p_L\in I^m(E)$ such
that the following holds:
\begin{multline}\label{tangent_lemma_new_random_label}
  \text{There exist }c,\delta,r>0 \text{
such that} \\ \abs{p_1(\vec x)}+\abs{p_2(\vec x)}+\dots+\abs{p_L(\vec x)}>
c\cdot\abs{\vec
{x}}^m \text { for all }\vec x\in\Gamma(\omega,\delta,r).
\end{multline}
By Definition \ref{defintion_of_I_m_of_E} for $1\leq l\leq L$ there exists $f_l\in C^m(\mathbb{R}^n)$
such that $f_l(\vec x)=p_l(\vec x)+r_l(\vec x)$, where \begin{gather}\label{little_o_for_q}r_l(\vec
x)=o(\abs{\vec
x}^m)\end{gather}and $f_{l}|_E=0$. As $\omega\in T(E)$ there exists a sequence
of points 
\begin{gather}\label{new_sequence_that_goes_to_the_origin}\{\vec x_i\}_{i=1}^{\infty}\subset
\Gamma(\omega,\delta,r)\cap E\text{ converging to the origin.}\end{gather}
On any of these points we have 
\begin{gather}\label{sequence_that_goes_to_the_origin}0=\abs{f_{1}(\vec x_i)}+\dots+\abs{f_{L}(\vec x_i)}=\abs{p_{1}(\vec
x_i)+r_{1}(\vec x_i)}+\dots+\abs{p_{L}(\vec
x_i)+r_{L}(\vec x_i)},\end{gather} and dividing (\ref{sequence_that_goes_to_the_origin})
by $\abs{x_i}^m\neq 0$ we get 
\begin{multline}\label{tangent_lemma_random_label}0= \frac{\abs{p_{1}(\vec
x_i)+r_{1}(\vec x_i)}+\dots+\abs{p_{L}(\vec
x_i)+r_{L}(\vec x_i)}}{\abs{\vec x_i}^m} \\ \geq \frac{\abs{p_{1}(\vec x_i)}+\dots+\abs{p_{L}(\vec x_i)}}{\abs{\vec
x_i}^m}-\frac{\abs{r_{1}(\vec x_i)}+\dots+\abs{r_{L}(\vec
x_i)}}{\abs{\vec
x_i}^m}.\end{multline} Now (\ref{tangent_lemma_random_label}) is a contradiction,
as (\ref{tangent_lemma_new_random_label}) and (\ref{new_sequence_that_goes_to_the_origin})
imply that $\frac{\abs{p_{1}(\vec x_i)}+\dots+\abs{p_{L}(\vec
x_i)}}{\abs{\vec
x_i}^m}$ is bounded from below by $c$, while (\ref{little_o_for_q}) and (\ref{new_sequence_that_goes_to_the_origin})
imply that $\frac{\abs{r_{1}(\vec x_i)}+\dots+\abs{r_{L}(\vec
x_i)}}{\abs{\vec
x_i}^m}$ goes to zero as $\vec x_i$ approaches the origin.\end{proof}

\begin{lemma}\label{relative_of_tangent_subset_bad_lemma}Let $\vec 0\in E\subset\mathbb{R}^n$
be a closed subset, let $m\in\mathbb{N}$ and denote $\Omega=Allow(I^m(E))$. Then, given $\delta>0$ there exists $\bar r>0$ such that \begin{gather}\label{relative_of_tangent_lemma_label_1}E\cap B^\times(\bar r)\subset \{\vec x\in B^\times(\bar r):\text{dist}(\frac{\vec x}{\abs{\vec x}},\Omega)<\delta\}\text{ if }\Omega\neq\emptyset\text{ and } E\cap
B^\times(\bar r)=\emptyset\text{ if }\Omega=\emptyset. \end{gather} \end{lemma}
\begin{proof}Assume $\Omega\neq\emptyset$ and assume the lemma does not hold. Then, there exists $\delta>0$ and a sequence of points $\{\vec x_i\}_{i=1}^{\infty}\subset E\setminus\{\vec0\}$ converging to the origin, such that $\text{dist}(\frac{\vec x_i}{\abs{\vec x_i}},\Omega)\geq \delta$, for any $i\in\mathbb{N}$. Since $S^{n-1}$ is compact, by possibly diluting the sequence $\{\vec x_i\}_{i=1}^{\infty}$ we may assume that $\frac{\vec x_i}{\abs{\vec x_i}}\to \omega^*\in S^{n-1}$ as $i\to\infty$. In particular, $\text{dist}(\omega^*,\Omega)\geq \delta$ and so $\omega^*\in Forb(I^m(E))$. By Lemma \ref{tangent_subset_bad_lemma}, $\omega^*\notin T(E)$ so there exist $\delta^*,r^*>0$ such that $$E\cap\Gamma(\omega^*,\delta^*,r^*)=\emptyset.$$ Since $\vec x_i$ is converging to the origin we may assume $0<\abs{\vec x_i}< r^*$ for large enough $i$, and since $\frac{\vec
x_i}{\abs{\vec x_i}}\to \omega^*$ we may assume $\abs{\frac{\vec x_i}{\abs{\vec x_i}}-\omega^*}<\delta^*$ for large enough $i$. Combining these two we have that for large enough $i$, $$\vec x_i\in \Gamma(\omega^*,\delta^*,r^*).$$But $\vec x_i\in E$, so we have $$E\cap\Gamma(\omega^*,\delta^*,r^*)\neq\emptyset,$$which is a contradiction. The proof of the case $\Omega=\emptyset$ is almost identical but slightly simpler (essentially repeat the proof but omit the condition "$\text{dist}(\frac{\vec
x_i}{\abs{\vec x_i}},\Omega)\geq \delta$, for any $i\in\mathbb{N}$" in the very beginning), thus we leave it to the reader to verify the details. \end{proof}

\subsection{Closed ideals -- a fundamental property of ideals of the form $I^m(E)$}

Recall that (when $n$ is clear from the context) for $K\geq 1$ and $
r>0$ we set $$\text{Ann}_K(r):=\{\vec x\in\mathbb{R}^n:\frac{r}{K}<\abs{\vec
x}<Kr\}.$$

\begin{definition}[implied jets]\label{new_def_implied_jet}Let $I\lhd\mathcal{P}_0^{m}(\mathbb{R}^n)$
be an ideal, denote $\Omega:=Allow(I)$ and let $p\in\mathcal{P}_0^{m}(\mathbb{R}^n)$ be some polynomial. We
say that\textit{ $I$ implies $p$} (or that $p$ is implied by $I$) if there exist a constant $A>0$ and $Q_1,Q_2,\dots, Q_L\in I$ such that the following holds:

For any $\epsilon>0$ there exist $\delta,r>0$ such that for any $0<\rho\leq r$ there exist functions $F,S_1,S_2,\dots,S_L\in C^m(\text{Ann}_4(\rho))$ satisfying \begin{gather}\label{new_def_of_implied_label_1}\abs{\partial^\alpha F(\vec x)}\leq \epsilon \rho^{m-\abs{\alpha}} \text { for all }\vec x\in\text{Ann}_4(\rho)
\text{ and all }\abs{\alpha}\leq m; \end{gather}
\begin{gather}\label{new_def_of_implied_label_2}\abs{\partial^\alpha S_l(\vec
x)}\leq A \rho^{-\abs{\alpha}} \text { for all }\vec x\in\text{Ann}_4(\rho)
\text{, all }\abs{\alpha}\leq m \text{ and all }1\leq l\leq L; \end{gather}
\begin{multline}\label{new_def_of_implied_label_3}p(\vec x)=F(\vec x)+S_1(\vec x)Q_1(\vec x)+S_2(\vec
x)Q_2(\vec x)+\dots+S_L(\vec
x)Q_L(\vec x) \\ \text {for all }\vec x\in\text{Ann}_2(\rho)
\text{ such that }\text{dist}(\frac{\vec x}{\abs{\vec x}},\Omega)<\delta. \end{multline}
\end{definition}

\begin{remark}Definition \ref{new_def_implied_jet} asserts the existence of functions $S_1,\dots,S_L$ for some $L\in\mathbb{N}$. One can define "1-implied jets" by only allowing $L=1$, so for instance (\ref{new_def_of_implied_label_3}) becomes "$p(\vec x)=F(\vec x)+S_1(\vec x)Q_1(\vec x)$ for all $\vec x\in\text{Ann}_2(\rho)$ such that $\text{dist}(\frac{\vec x}{\abs{\vec x}},\Omega)<\delta$". Clearly if a jet is 1-implied by an ideal it is also implied by the same ideal. We do not know whether the converse also holds; in particular we are not able to construct an ideal $I$ and a jet $p\in\mathcal{P}_0^{m}(\mathbb{R}^n)$, such that $p$ is implied by $I$, but $p$ is not 1-implied by $I$. Similar remarks can be made regarding Definition \ref{implied_polynomial_definition} and Definition \ref{strong_implied_polynomial_definition}.\end{remark}

\begin{definition}[the closure of an ideal]\label{the_closure_of_ideal_definition}Let
$I\lhd\mathcal{P}_0^m(\mathbb{R}^n)$ be an ideal. We define its \textit{implication
closure} (or simply \textit{closure}) $cl(I)$ by $$cl(I):=\{p\in\mathcal{P}_0^m(\mathbb{R}^n)|p\text{
is implied by }I\}.$$ We say that $I$ is \textit{closed} if $I=cl(I)$.\end{definition}

\begin{remark}Let $p\in I\lhd\mathcal{P}_0^m(\mathbb{R}^n)$. Then, we can
take $A=1, l=1, Q_1=p, S_1=1$ and $F=0$, and for any $\epsilon>0$ take any $\delta,r>0$ to prove $p\in cl(I)$. Thus, we always
have $I\subset cl(I)$. Moreover, the closure of any ideal is an ideal as well.  
\end{remark}

\begin{lemma}\label{lemma_needed_for_invariance_wrt_coor_changes}Let $I\lhd\mathcal{P}_0^{m}(\mathbb{R}^n)$
be an ideal, denote $\Omega:=Allow(I)$ and let $p\in\mathcal{P}_0^{m}(\mathbb{R}^n)$
be some polynomial that is implied by $I$. Then, given $K\geq4$ there exist
a constant $\tilde A>0$ and $Q_1,Q_2,\dots, Q_L\in I$ such that the following
holds:

For any $\epsilon>0$ there exist $\delta,r>0$ such that for any $0<\rho\leq
r$ there exist functions $F,S_1,S_2,\dots,S_L\in C^m(\mathbb{R}^n)$
satisfying \begin{gather}\label{lemma_for_invariance_label_1}\abs{\partial^\alpha
F(\vec x)}\leq \epsilon \rho^{m-\abs{\alpha}} \text { for all }\vec x\in\mathbb{R}^n
\text{ and all }\abs{\alpha}\leq m; \end{gather}
\begin{gather}\label{lemma_for_invariance_label_2}\abs{\partial^\alpha S_l(\vec
x)}\leq \tilde A \rho^{-\abs{\alpha}} \text{ for all }\vec x\in\mathbb{R}^n
\text{, all }\abs{\alpha}\leq m \text{ and all }1\leq l\leq L; \end{gather}
\begin{multline}\label{lemma_for_invariance_label_3}p(\vec x)=F(\vec x)+S_1(\vec x)Q_1(\vec x)+S_2(\vec
x)Q_2(\vec x)+\dots+S_L(\vec
x)Q_L(\vec x) \\ \text {for all }\vec x\in\text{Ann}_K(\rho)
\text{ such that }\text{dist}(\frac{\vec x}{\abs{\vec x}},\Omega)<\delta.
\end{multline}\end{lemma}
\begin{proof}Let  $A>0$ and $Q_1,Q_2,\dots, Q_L\in I$ be as in Definition
\ref{new_def_implied_jet}, and let $K\geq4$. Fix $\epsilon>0$ and set $\tilde
\epsilon:=\frac{\epsilon}{2C''}$, where $C''$ is a constant depending only
on $m,n,K$ and $A$, to be determined below. Corresponding to $\tilde\epsilon$, let $\delta,r>0$
be as in Definition
\ref{new_def_implied_jet}, and set $\tilde r=\frac{r}{1000K}$. Suppose $0<\rho\leq\frac{r}{1000 K}$. We introduce a
partition of unity $\{\theta_\nu\}_{\nu=1,\dots,\nu_{\max}}$, satisfying:
\begin{itemize}
\item For all $\nu=1,\dots,\nu_{\max}$, $\theta_\nu\in C^\infty(\mathbb{R}^n)$ and $\text{supp}\theta_\nu\subset \text{Ann}_2(\rho_\nu)$ for some $\frac{\rho}{1000K}\leq\rho_\nu\leq1000K\rho$. 
\item $\nu_{\max}\leq C$;
\item $1=\sum\limits_{\nu=1}^{\nu_{\max}}\theta_\nu(\vec x)$ for any $\vec x\in\text{Ann}_{40K}(\rho)$;
\item $\abs{\partial^\alpha\theta_\nu(\vec x)}\leq C'\rho_\nu^{-\abs{\alpha}}\leq(1000K)^mC'\rho^{-\abs{\alpha}}$ for any $\vec x\in\mathbb{R}^n$, $\nu=1,\dots,\nu_{\max}$ and  $\abs{\alpha}\leq m$, 
\end{itemize}where above $C$ and $C'$ are positive constants depending only on $m,n,K$ and $A$. By Definition
\ref{new_def_implied_jet} for each $\nu=1,\dots,\nu_{\max}$ there exist functions $F^\nu,S_1^\nu,S_2^\nu,\dots,S_L^\nu\in C^m(\text{Ann}_4(\rho_\nu))$
satisfying \begin{gather}\label{lemma_for_invariance_proof_label_1}\abs{\partial^\alpha
F^\nu(\vec x)}\leq \tilde\epsilon \rho_\nu^{m-\abs{\alpha}}\leq  (1000K)^m\tilde\epsilon \rho^{m-\abs{\alpha}} \text { for all }\vec x\in\text{Ann}_4(\rho_\nu)
\text{ and all }\abs{\alpha}\leq m; \end{gather}
\begin{gather}\label{lemma_for_invariance_proof_label_2}\abs{\partial^\alpha S^\nu_l(\vec
x)}\leq A \rho_\nu^{-\abs{\alpha}}\leq(1000K)^mA\rho^{-\abs{\alpha}} \text { for all }\vec x\in\text{Ann}_4(\rho)
\text{, }\abs{\alpha}\leq m \text{ and all }1\leq l\leq L; \end{gather}
\begin{multline}\label{lemma_for_invariance_proof_label_3}p(\vec x)=F^\nu(\vec x)+S_1^\nu(\vec x)Q_1(\vec x)+S_2^\nu(\vec
x)Q_2(\vec x)+\dots+S_L^\nu(\vec
x)Q_L(\vec x) \\ \text {for all }\vec x\in\text{Ann}_2(\rho_\nu)
\text{ such that }\text{dist}(\frac{\vec x}{\abs{\vec x}},\Omega)<\delta.
\end{multline} Define $F=\sum\limits_{\nu=1}^{\nu_{\max}}\theta_\nu F^\nu\in C^m(\mathbb{R}^n)$ and for $l=1,\dots,L$ define $S_{l}=\sum\limits_{\nu=1}^{\nu_{\max}}\theta_\nu S_l^\nu\in
C^m(\mathbb{R}^n)$. Then, (\ref{lemma_for_invariance_proof_label_1})-(\ref{lemma_for_invariance_proof_label_3}) together with the properties of $\{\theta_\nu\}_{\nu=1,\dots,\nu_{\max}}$ imply that \begin{gather}\label{lemma_for_invariance_proof_label_4}\abs{\partial^\alpha
F(\vec x)}\leq C''\tilde \epsilon \rho^{m-\abs{\alpha}}=\frac{\epsilon}{2} \rho^{m-\abs{\alpha}}< \epsilon
\rho^{m-\abs{\alpha}} \text { for all
}\vec x\in\mathbb{R}^n
\text{ and all }\abs{\alpha}\leq m; \end{gather}
\begin{gather}\label{lemma_for_invariance_proof_label_5}\abs{\partial^\alpha
S_l(\vec
x)}\leq \tilde A \rho_\nu^{-\abs{\alpha}} \text
{ for all }\vec x\in\mathbb{R}^n
\text{, }\abs{\alpha}\leq m \text{ and all }1\leq l\leq L; \end{gather}
\begin{multline}\label{lemma_for_invariance_proof_label_6}p(\vec x)=\sum\limits_{\nu=1}^{\nu_{\max}} \theta_\nu(\vec x) p(\vec x)\\=\sum\limits_{\nu=1}^{\nu_{\max}} \theta_\nu(\vec x)[F^\nu(\vec x)+S_1^\nu(\vec x)Q_1(\vec x)+S_2^\nu(\vec
x)Q_2(\vec x)+\dots+S_L^\nu(\vec
x)Q_L(\vec x)]\\=F(\vec x)+S_1(\vec x)Q_1(\vec x)+S_2(\vec
x)Q_2(\vec x)+\dots+S_L(\vec
x)Q_L(\vec x) \\ \text {for all }\vec x\in\text{Ann}_K(\rho)
\text{ such that }\text{dist}(\frac{\vec x}{\abs{\vec x}},\Omega)<\delta,
\end{multline} where above $C''$ and $\tilde A$ are positive constants depending only
on $m,n,K$ and $A$. 

Finally, for a fixed $\epsilon>0$ we let $\delta,\tilde r>0$ be as above. So for a given $0<\rho \leq \tilde r$ we have produced functions $F,S_1,S_2,\dots,S_L\in C^m(\mathbb{R}^n)$
satisfying (\ref{lemma_for_invariance_proof_label_4})-(\ref{lemma_for_invariance_proof_label_6}), which are equivalent to (\ref{lemma_for_invariance_label_1})-(\ref{lemma_for_invariance_label_3}), and so the lemma holds.\end{proof}

\begin{lemma}Definition \ref{new_def_implied_jet}
is invariant with respect
to $C^m$ coordinate changes around the origin, namely: let $I\lhd\mathcal{P}_0^{m}(\mathbb{R}^n)$
be an ideal, let $p\in\mathcal{P}_0^{m}(\mathbb{R}^n)$
be some polynomial, and assume $I$ implies $p$. Let $\phi:\mathbb{R}^n\to\mathbb{R}^n$  be  a $C^m$-diffeomorphism that
fixes the origin, recall it induces an automorphism of $\mathcal{P}^m(\mathbb{R}^n)$
defined by $p\mapsto J^{m}(p\circ\phi)$, and denote this automorphism by $\phi^*$.  Then, $\phi^*(I)$ implies $\phi^*(p)$. In particular, the property of an ideal being
closed is invariant with respect
to $C^m$ coordinate changes around the origin.\end{lemma}

\begin{proof}The image of any cone under any
$C^m$ coordinate change around the origin contains a cone. Thus, the notion
of allowed directions of an ideal is invariant under $C^m$ coordinate changes
around
the origin, up to a linear coordinate change (that arises from the Jacobian
of the $C^m$ coordinate change). The image of any annulus centered at the origin under any
$C^m$ coordinate change around the origin is contained in an annulus centered
at the origin. These two facts together with Lemma \ref{lemma_needed_for_invariance_wrt_coor_changes} prove the lemma. We leave it to the reader to verify the details. \end{proof}

\begin{remark}Another
invariant of a given ideal $I$ with respect to $C^m$
coordinate changes is $\dim\text{span}_{\mathbb{R}}Allow(I)$ -- we will use
this fact often in \cite{FS2}.\end{remark}

\begin{corollary}[sets with no allowed directions]\label{cor_no_allowed_directions}The
only closed ideal $I\lhd\mathcal{P}_0^m(\mathbb{R}^n)$ such that $Allow(I)=\emptyset$
is $I=\mathcal{P}_0^m(\mathbb{R}^n)$. Moreover, if $E\subset\mathbb{R}^n$
is a closed subset containing the origin such that $Allow(I^m(E))=\emptyset$,
then the origin is an isolated
point of $E$ and $I^m(E)=\mathcal{P}_0^m(\mathbb{R}^n)$.\end{corollary}

\begin{proof}  If $I$ is closed and $Allow(I)=\emptyset$ then any jet
in $\mathcal{P}_0^m(\mathbb{R}^n)$
is implied by $I$ as (\ref{new_def_of_implied_label_3}) vacuously holds (to satisfy (\ref{new_def_of_implied_label_1}) and (\ref{new_def_of_implied_label_2}) one
may take $A=1, l=1, Q_1=0, S_1=0$ and $F=0$). This proves that the
only closed ideal $I\lhd\mathcal{P}_0^m(\mathbb{R}^n)$ such that $Allow(I)=\emptyset$
is $I=\mathcal{P}_0^m(\mathbb{R}^n)$. The "moreover" part  follows from Lemma \ref{relative_of_tangent_subset_bad_lemma}, Observation \ref{simple_observations}(i) and the fact that $I^m(\{\vec 0\})=\mathcal{P}_0^m(\mathbb{R}^n)$. \end{proof}

\begin{theorem}\label{main_theorem_on_necessary_condition}Fix $m,n\in\mathbb{N}$
and let $E\subset\mathbb{R}^n$ be a closed subset containing the origin.
Then, $I^m(E)\lhd\mathcal{P}_0^{m}(\mathbb{R}^n)$ is a closed ideal.\end{theorem}

\begin{proof}Denote $\Omega=Allow(I^m(E))$. If $\Omega=\emptyset$ then the theorem follows from Corollary \ref{cor_no_allowed_directions}. So we assume $\Omega\neq\emptyset$. Let $p\in\mathcal{P}_0^{m}(\mathbb{R}^n)$ be implied by $I^m(E)$, and let $A>0$ and $Q_1,\dots,Q_L\in I^m(E)$ be such that the upshot of Definition \ref{new_def_implied_jet} holds. For $k\in\mathbb{N}$ set $\epsilon_k=2^{-k}$ and let $\delta_k,r_k$ correspond to $\epsilon_k$ as in Definition
\ref{new_def_implied_jet}. Given $0<\rho\leq r_k$ there exist functions $F,S_1,S_2,\dots,S_L\in C^m(\text{Ann}_4(\rho))$ such that
\begin{gather}\label{new_proof_main_theorem_label_1}\abs{\partial^\alpha
F(\vec x)}\leq \epsilon_{k} \rho^{m-\abs{\alpha}} \text { for all }\vec x\in\text{Ann}_4(\rho)
\text{ and all }\abs{\alpha}\leq m; \end{gather}
\begin{gather}\label{new_proof_main_theorem_label_2}\abs{\partial^\alpha S_l(\vec
x)}\leq A \rho^{-\abs{\alpha}} \text { for all }\vec x\in\text{Ann}_4(\rho)
\text{, all }\abs{\alpha}\leq m \text{ and all }1\leq l\leq L; \end{gather}
\begin{multline}\label{new_proof_main_theorem_label_3}p(\vec x)=F(\vec x)+S_1(\vec
x)Q_1(\vec x)+S_2(\vec
x)Q_2(\vec x)+\dots+S_L(\vec
x)Q_L(\vec x) \\ \text {for all }\vec x\in\text{Ann}_2(\rho)
\text{ such that }\text{dist}(\frac{\vec x}{\abs{\vec x}},\Omega)<\delta_{k}.
\end{multline}
By Lemma \ref{relative_of_tangent_subset_bad_lemma} there exists $\bar r_k>0$ such that
\begin{gather}\label{new_proof_main_theorem_label_4}E\cap B^\times(4 \bar r_k)\subset \{\vec x\in B^\times(4\bar r_{k}):\text{dist}(\frac{\vec
x}{\abs{\vec x}},\Omega)<\delta_{k}\}. \end{gather}
Replacing $\bar r_k$ by a smaller number preserves (\ref{new_proof_main_theorem_label_4}), so without loss of generality we may assume that for any $k\in \mathbb{N}$: $\bar r_k$ is a negative integer power of $2$; $\bar r_k<r_k$ and $\bar r_{k+1}\leq 2^{-10}\bar r_k$. In particular $\bar r_k\to 0$ as $k\to\infty$.

Set $\nu_{min}$ to be the unique integer such that $\bar r_1=2^{-\nu_{min}}$.
For $\mathbb{N}\ni \nu\geq\nu_{min}$ set $k(\nu)$ to be  the unique integer  such that
$\bar r_{k(\nu)+1}<2^{-\nu}\leq \bar r_{k(\nu)}$. Then, for any $\nu\geq\nu_{min}$
we in particular have $2^{-\nu}\leq\bar r_{k(\nu)}<r_{k(\nu)}$, so setting $\rho:=2^{-\nu}$ in (\ref{new_proof_main_theorem_label_1})-(\ref{new_proof_main_theorem_label_3}) above we obtain functions $F^\nu,S^\nu_1,S^\nu_2,\dots,S^\nu_L\in
C^m(\text{Ann}_4(2^{-\nu}))$ such that
\begin{gather}\label{new_proof_main_theorem_label_5}\abs{\partial^\alpha
F^{\nu}(\vec x)}\leq \epsilon_{k(\nu)} (2^{-\nu})^{m-\abs{\alpha}} \text { for all }\vec x\in\text{Ann}_4(2^{-\nu})
\text{ and all }\abs{\alpha}\leq m; \end{gather}
\begin{gather}\label{new_proof_main_theorem_label_6}\abs{\partial^\alpha
S^{\nu}_l(\vec
x)}\leq A (2^{-\nu})^{-\abs{\alpha}} \text{ for all }\vec x\in\text{Ann}_4(2^{-\nu})
\text{, all }\abs{\alpha}\leq m \text{ and all }1\leq l\leq L; \end{gather}
\begin{multline}\label{new_proof_main_theorem_label_7}p(\vec x)=F^{\nu}(\vec x)+S^{\nu}_1(\vec
x)Q_1(\vec x)+S^{\nu}_2(\vec
x)Q_2(\vec x)+\dots+S^{\nu}_L(\vec
x)Q_L(\vec x) \\ \text {for all }\vec x\in\text{Ann}_2(2^{-\nu})
\text{ such that }\text{dist}(\frac{\vec x}{\abs{\vec x}},\Omega)<\delta_{k(\nu)}.
\end{multline} Since $2\cdot 2^{-\nu}<4\bar r_{k(\nu)}$, we have that (\ref{new_proof_main_theorem_label_4}) and (\ref{new_proof_main_theorem_label_7}) imply that \begin{multline}\label{new_proof_main_theorem_label_8}p(\vec x)=F^{\nu}(\vec
x)+S^{\nu}_1(\vec
x)Q_1(\vec x)+S^{\nu}_2(\vec
x)Q_2(\vec x)+\dots+S^{\nu}_L(\vec
x)Q_L(\vec x) \\ \text {for all }\vec x\in E\cap \text{Ann}_2(2^{-\nu}).
\end{multline} Note that $k(\nu)\to\infty$ as $\nu\to\infty$, and so
\begin{gather}\label{new_proof_main_theorem_label_10}\epsilon_{k(\nu)}=2^{-k(\nu)}\to0\text{ as }\nu\to\infty. \end{gather} Let $\{\theta_\nu\in C^m(\mathbb{R}^n)\}_{\nu\geq\nu_{min}}$ be such that for some constant $C_1>0$ (that depends only on $m$ and $n$) we have
\begin{gather}\label{new_proof_main_theorem_label_11}0\leq\theta_\nu(\vec x)\leq 1\text{ for all }\vec x\in\mathbb{R}^n \text{ and }\text{supp}(\theta_\nu)\subset\text{Ann}_2(2^{-\nu})
\text{ for all }\nu\geq\nu_{min}; \end{gather}
\begin{gather}\label{new_proof_main_theorem_label_12}\abs{\partial^\alpha
\theta_\nu(\vec x)}\leq C_1(2^{-\nu})^{-\abs{\alpha}} \text
{ for
all }\vec x\in\mathbb{R}^n, 
\text{ all }\abs{\alpha}\leq m\text{ and all }\nu\geq\nu_{min}; \end{gather}
\begin{gather}\label{new_proof_main_theorem_label_13}\sum\limits_{\nu\geq\nu_{min}}\theta_\nu(\vec x)=1\text
{ for
all }0<\abs{\vec x}\leq 2^{-10}\cdot 2^{-\nu_{min}}.\end{gather}We now define $F,S_1,S_2,\dots,S_L\in C^m(\mathbb{R}^n\setminus\{\vec 0\})$ by $$F(\vec x):=\sum\limits_{\nu\geq\nu_{min}}\theta_\nu(\vec
x)F^\nu(\vec x)\text{ and }S_{l}(\vec
x):=\sum\limits_{\nu\geq\nu_{min}}\theta_\nu(\vec
x)S_{l}^\nu(\vec x).$$Fix an integer $\mu$ and suppose $2^{-(\mu+1)}<\abs{\vec x}\leq2^{-\mu}$. Then, $\vec x\in \text{supp}(\theta_\nu)=\text{Ann}_2(2^{-\nu})$ (recall (\ref{new_proof_main_theorem_label_11})) only for $\nu$ satisfying $\abs{\nu-\mu}\leq2$ since $2^{-(\mu+1)}<\abs{\vec
x}<2\cdot 2^{-\nu}$ and $\frac{1}{2}\cdot 2^{-\nu}<\abs{\vec
x}<2^{-\mu}$. Consequently, (\ref{new_proof_main_theorem_label_5}) and (\ref{new_proof_main_theorem_label_12}) yield that for  some constants $C_2,C_3>0$ (that depend only on $m$ and $n$) we have
\begin{multline}\label{new_proof_main_theorem_label_14}\abs{\partial^\alpha
F(\vec x)}\leq \sum\limits_{\abs{\nu-\mu}\leq2;\nu\geq\nu_0}C_2\epsilon_{k(\nu)} (2^{-\nu})^{m-\abs{\alpha}}\leq  C_{3}\big[\sum\limits_{\abs{\nu-\mu}\leq2;\nu\geq\nu_0}\epsilon_{k(\nu)}
\big](2^{-\mu})^{m-\abs{\alpha}} \\ \text
{ for all }2^{-(\mu+1)}<\abs{\vec x}\leq 2^{-\mu},\abs{\alpha}\leq m\text{ and all }\mu\in\mathbb{Z}.\end{multline}
For a constant $C_4>0$ (that depends only on $m$ and
$n$) we have also
\begin{gather}\label{new_proof_main_theorem_label_15}\abs{\partial^\alpha
S_l(\vec
x)}\leq C_{4}A (2^{-\nu})^{-\abs{\alpha}} \text
{ for }2^{-\nu-1}<\abs{\vec x}\leq 2^{-\nu},\text{ for all }\abs{\alpha}\leq
m\text{ and all }\nu\geq\nu_{min}.\end{gather}Recall that $\bar r_1=2^{-\nu_{min}}$, and (\ref{new_proof_main_theorem_label_10}), (\ref{new_proof_main_theorem_label_14}) and (\ref{new_proof_main_theorem_label_15}) imply that for some constant $C_5>0$ (that depends only on $m$ and
$n$) we have \begin{gather}\label{new_proof_main_theorem_label_16}\abs{\partial^\alpha
F(\vec
x)}=o(\abs{\vec x}^{m-\alpha})\text{ as }\abs{\vec x}\to0\text{ for all }\abs{\alpha}\leq
m,\end{gather}
\begin{gather}\label{new_proof_main_theorem_label_17}\abs{\partial^\alpha
S_l(\vec
x)}\cdot \abs{\vec x}^{\abs{\alpha}}<C_5 A \text
{ for all }\vec x\in B^\times(2^{-10}\cdot \bar r_1)\text{ and all }\abs{\alpha}\leq
m.\end{gather}Now, (\ref{new_proof_main_theorem_label_8}), (\ref{new_proof_main_theorem_label_11}) and (\ref{new_proof_main_theorem_label_13}) imply that \begin{multline}\label{new_proof_main_theorem_label_18}p(\vec
x)=F(\vec
x)+S_1(\vec
x)Q_1(\vec x)+S_2(\vec
x)Q_2(\vec x)+\dots+S_L(\vec
x)Q_L(\vec x) \\ \text {for all }\vec x\in E\cap B^\times(2^{-10}\cdot \bar r_1).\end{multline}Let $\chi\in C^\infty(\mathbb{R}^n)$ be such that $\chi$ identically equals 1 on $B^\times(2^{-11}\cdot \bar
r_1)$ and $\text{supp}(\chi)\subset(2^{-10}\cdot \bar
r_1)$. Then, (\ref{new_proof_main_theorem_label_18}) implies that
\begin{multline}\label{new_proof_main_theorem_label_19}p(\vec
x)=[(1-\chi(\vec x))\cdot p(\vec x)+\chi(\vec x)\cdot F(\vec x)]\\+\chi(\vec x)\cdot [S_1(\vec
x)Q_1(\vec x)+S_2(\vec
x)Q_2(\vec x)+\dots+S_L(\vec
x)Q_L(\vec x)] \\ \text {for all }\vec x\in E\setminus \{\vec 0\}.\end{multline}
We now define
$\tilde F\in C^m(\mathbb{R}^n)$ (thanks to (\ref{new_proof_main_theorem_label_16})) and $\tilde S_1,\tilde S_2,\dots,\tilde S_L\in C^m(\mathbb{R}^n\setminus\{\vec 0\})$ by $$\tilde F:=\begin{cases}
  (1-\chi(\vec x))\cdot p(\vec x)+\chi(\vec x)\cdot F(\vec x)  & \text{if } \vec x\neq\vec 0 \\
  0 & \text{if } \vec x=\vec 0
\end{cases}\text{ and }\tilde S_l(\vec x):=\chi(\vec
x)\cdot S_l(\vec
x),$$and thanks to (\ref{new_proof_main_theorem_label_19}), (\ref{new_proof_main_theorem_label_16}) and (\ref{new_proof_main_theorem_label_17}) (respectively) we get 
\begin{gather}\label{new_proof_main_theorem_label_20}p(\vec
x)=\tilde F(\vec x)+\tilde S_1(\vec
x)Q_1(\vec x)+\tilde S_2(\vec
x)Q_2(\vec x)+\dots+\tilde S_L(\vec
x)Q_L(\vec x) \text { for all }\vec x\in E\setminus \{\vec 0\};\end{gather}
\begin{gather}\label{new_proof_main_theorem_label_21}J^m(\tilde F)=0;\end{gather}
\begin{gather}\label{new_proof_main_theorem_label_22}\abs{\partial^\alpha \tilde S_l(\vec x)}=O(\abs{\vec x}^{-\alpha}) \text{ for all }\abs{\alpha}\leq m \text{ and all }1\leq l\leq L.\end{gather}We recall that for any $1\leq l\leq L$, $Q_l\in I^m(E)$, so there exists $F_l^\#\in C^m(\mathbb{R}^n)$ with $J^m(F_l^\#)=0$, such that $Q_l+F_l^\#=0$ on $E$. So now (\ref{new_proof_main_theorem_label_20}) implies that \begin{multline}\label{new_proof_main_theorem_label_23}p(\vec
x)=[\tilde F(\vec x)-\sum\limits_{l=1}^L\tilde S_l(\vec x)F_l^\#(\vec x)]+\sum\limits_{l=1}^L\tilde S_l(\vec x)[Q_l(\vec x)+F_l^\#(\vec x)] \\ =[\tilde F(\vec x)-\sum\limits_{l=1}^L\tilde S_l(\vec x)F_l^\#(\vec x)] \text { for all }\vec x\in E\setminus \{\vec 0\}.\end{multline}Finally, we define $$F^\#(\vec x):=\begin{cases}
  \sum\limits_{l=1}^L\tilde
S_l(\vec x)F_l^\#(\vec x)-\tilde F(\vec x)  & \text{if
} \vec x\neq\vec 0 \\
  0 & \text{if } \vec x=\vec 0
\end{cases}.$$ 
Since $\tilde F,F_1^\#,F_2^\#,\dots F_L^\#\in C^m(\mathbb{R}^n)$ and $J^m(\tilde F),J^m((F_1^\#),J^m(F_2^\#),\dots J^m(F_L^\#)=0$, we get from (\ref{new_proof_main_theorem_label_22}) that $F^\#\in C^m(\mathbb{R}^n)$ and $J^m(F^\#)=0$. Moreover, from (\ref{new_proof_main_theorem_label_23}) we get that $p+F^\#$ vanishes identically on $E$. We conclude that $p\in I^m(E)$, i.e., $I^m(E)$ is indeed closed.\end{proof}

\begin{question}[being closed is sufficient]\label{main_question}Fix $m,n\in\mathbb{N}$.
Is it true that for any closed ideal $I\lhd\mathcal{P}_0^{m}(\mathbb{R}^n)$ there
exists a closed subset $E\subset\mathbb{R}^n$ containing the origin  such
that $I=I^m(E)$?\end{question}

\begin{question}[semi-algebraic sets suffice]\label{semi_algebraic_question}Fix
$m,n\in\mathbb{N}$ and let $E\subset\mathbb{R}^n$ be a
closed subset containing the origin. Is it true that there always exists
a semi-algebraic subset $E'\subset\mathbb{R}^n$ containing the origin (see Footnote \ref{semialgebraic_footnote}) such
that $I^m(E)=I^m(E')$?\end{question}

\section{Calculating implied jets}\label{section_calculation_method}
\subsection{Calculating the allowed directions of an ideal}

\begin{definition}[lowest degree homogenous part]Let $p\in\mathcal{P}_0^{m}(\mathbb{R}^n)$
be a non-zero jet. We may always uniquely write $p=p_k+q$,
with $p_k$ a homogenous polynomial
of degree $k$, where $k$ is the order of vanishing
of $p$ (as $p$ is not the zero jet we have $1\leq k\leq m$ -- see Definition
\ref{order-of-vanishing-def}), and $q$ is a (possibly zero) polynomial of
order of vanishing more than 
$k$. We call $p_k$ \textit{the lowest degree homogenous part of} $p$.\end{definition}

The proof of the following lemma is straightforward, and so left to the reader.

\begin{lemma}\label{equiv_conditions_for_forbidden_directions_lemma}
Let $p\in\mathcal{P}_0^m(\mathbb{R}^n)$ be a non-zero jet and let $\omega\in
S^{n-1}$ be a direction. Let $p_k$ be the lowest degree
homogenous part of $p$. Then, the following are equivalent:
\begin{gather}\label{equiv_cond_two_for_forbidden}
  \text{ There exist a cone }\Gamma(\omega,\delta,r)\text{ and }c>0\text{
such that }\abs{p(\vec x)}> c\cdot\abs{\vec
{x}}^k\text { for all }\vec x\in\Gamma(\omega,\delta,r);
\end{gather}
\begin{gather}\label{equiv_cond_three_for_forbidden}
  p_k(\omega)\neq 0;
\end{gather} \end{lemma}
\begin{corollary}\label{old_new_cor_on_how_to_calc_allow_with_inclusion}Let
$I=\langle
p_1,p_2,\dots p_t\rangle_m\lhd\mathcal{P}^{m}(\mathbb{R}^n)$ be an ideal. Then,
$$Allow(I)\subset\bigcap\limits_{i=1}^t\{\text{the zero set in } S^{n-1}\text{
of the lowest degree homogenous part of }p_i\}.$$Moreover, if $
p_1,p_2,\dots p_t$ are all homogenous, then $$Allow(I)=\bigcap\limits_{i=1}^t\{\text{the
zero set in } S^{n-1}\text{
of }p_i\}.$$\end{corollary}

\begin{proof}The first part follows immediately from Lemma \ref{equiv_conditions_for_forbidden_directions_lemma}
and the fact that (\ref{equiv_cond_two_for_forbidden}) obviously implies
(\ref{equiv_cond_one_for_forbidden}). It is left to the reader to verify the "moreover" part as it now follows easily from Definition \ref{definition_allowed_and_forbidden_directions}. \end{proof}

\subsection{Negligible functions}

\begin{definition}[negligible functions]\label{def_negligible_functions}Let $U\subset\mathbb{R}^n$ be open and let $\Omega\subset S^{n-1}$. A function $F\in C^m(U)$
is \textit{(m--)negligible for $\Omega$ }if for all $\epsilon>0$ there
exist  $\delta>0$ and $r>0$ such that the following hold:
\begin{gather}\label{negligible_fun_condition_0}\Gamma(\Omega,\delta,r)\subset U;\end{gather}
\begin{gather}\label{negligible_fun_condition_a}\abs{\partial^\alpha
F(\vec x)}\leq\epsilon\abs{\vec x}^{m-\abs{\alpha}}\text { for all
}\vec x\in\Gamma(\Omega,\delta,r)\text{ and all }\abs{\alpha}\leq
m;\end{gather}
\begin{multline}\label{negligible_fun_condition_b}\abs{\partial^\alpha F(\vec x)-\sum\limits_{\abs{\gamma}\leq m-\abs{\alpha}}\frac{1}{\gamma!}\partial^{\alpha+\gamma}
F(\vec y)\cdot(\vec x-\vec y)^{\gamma}}\leq \epsilon\abs{\vec x-\vec
y}^{m-\abs{\alpha}} \\ \text { for all }\vec x,\vec y\in\Gamma(\Omega,\delta,r)
\text{ distinct, and all }\abs{\alpha}\leq m.\end{multline}
\end{definition}

\begin{example}[follows easily from Taylor's Theorem]\label{exmaple_easy_eaxample_of_negligible_functions}If there exists $r>0$ such that $B(r)\subset U$ and
$J^m(F)=0$, then $F$ is negligible for $\Omega$, for any $\Omega
\subset S^{n-1}$. \end{example}

\begin{definition}[Whitney-negligible functions]\label{def_whitney_negligible_functions}Let
$U\subset\mathbb{R}^n$ be open and let $\Omega\subset S^{n-1}$.
A function $F\in C^m(U)$
is \textit{Whitney--(m--)negligible for $\Omega$ }if for all $\epsilon>0$ there
exist  $\delta>0$, $r>0$ and $F_\epsilon\in C^m(\mathbb{R}^n\setminus\{\vec0\})$ such that the following hold:
\begin{gather}\label{whitney_negligible_fun_condition_0}\Gamma(\Omega,\delta,r)\subset
U;\end{gather}
\begin{gather}\label{whitney_negligible_fun_condition_1}\abs{\partial^\alpha
F_\epsilon(\vec x)}\leq\epsilon\abs{\vec x}^{m-\abs{\alpha}}\text { for all
}\vec x\in\mathbb{R}^n\setminus\{\vec0\}\text{ and all }\abs{\alpha}\leq
m;\end{gather}
\begin{gather}\label{whitney_negligible_fun_condition_2}F_\epsilon(\vec x)=F(\vec x)\text { for all
}\vec x\in\Gamma(\Omega,\delta,r).\end{gather}
\end{definition}

We will show that a function is negligible if and only if it is Whitney--negligible (Lemma \ref{lemma_euqiv_def_of_neg_fun}). Before we do that we state and prove a preliminary lemma that follows from Whitney's Extension Theorem:

\begin{lemma}\label{new_cor_from_whitney}Fix $m,n\in\mathbb{N}$. There exists a constant $C(m,n)>0$ depending only on $m$ and $n$ such that the following holds: 

Let $\emptyset\neq\Omega\subset S^{n-1}$, $\delta>0$, and $F\in C^m(U_\delta)$, where $U_\delta=\{\vec x\in\mathbb{R}^n:\frac{1}{10}<\abs{\vec x}<1,\text{dist}(\frac{\vec x}{\abs{\vec x}},\Omega)<\delta\}$. Also set $\tilde U_\delta=\{\vec
x\in\mathbb{R}^n:\frac{1}{2}\leq\abs{\vec x}\leq\frac{2}{3},\text{dist}(\frac{\vec
x}{\abs{\vec
x}},\Omega)\leq\frac{\delta}{2}\}$. Let $M>0$ be such that

\begin{gather}\label{whitney_ext_thm_label_1}\abs{\partial^\alpha
F(\vec x)}\leq M\text { for all
}\vec x\in U_\delta\text{ and all }\abs{\alpha}\leq
m;\end{gather}
\begin{multline}\label{whitney_ext_thm_label_2}\abs{\partial^\alpha F(\vec
x)-\sum\limits_{\abs{\gamma}\leq m-\abs{\alpha}}\frac{1}{\gamma!}\partial^{\alpha+\gamma}
F(\vec y)\cdot(\vec x-\vec y)^{\gamma}}\leq M\abs{\vec x-\vec
y}^{m-\abs{\alpha}} \\ \text { for all }\vec x,\vec y\in U_\delta
\text{ distinct, and all }\abs{\alpha}\leq m.\end{multline}
Then, there exists $\tilde F\in C^m(\mathbb{R}^n)$ such that
\begin{gather}\label{whitney_ext_thm_label_3}\abs{\partial^\alpha\tilde
F(\vec x)}\leq C(m,n)\cdot M\text { for all
}\vec x\in \mathbb{R}^n\text{ and all }\abs{\alpha}\leq
m;\end{gather}
\begin{gather}\label{whitney_ext_thm_label_4}\tilde F(\vec x)=F(\vec
x)\text { for all
}\vec x\in \tilde U_\delta.\end{gather}
\end{lemma}

\begin{proof}We start by checking that the following holds:\begin{multline}\label{whitney_ext_thm_label_5}\abs{\partial^\alpha F(\vec
x)-\sum\limits_{\abs{\gamma}\leq m-\abs{\alpha}}\frac{1}{\gamma!}\partial^{\alpha+\gamma}
F(\vec y)\cdot(\vec x-\vec y)^{\gamma}}=o(\abs{\vec x-\vec y}^{m-\abs{\alpha}}) \\ \text { as }\abs{\vec x-\vec y}\to 0,\text{ subject to }\vec x,\vec y\in \tilde U_\delta
\text{ distinct, and all }\abs{\alpha}\leq m.\end{multline}Indeed, suppose (\ref{whitney_ext_thm_label_5}) does not hold. Then, there exist a multi-index $\alpha_0$ with $\abs{\alpha_0}\leq m$, a number $\eta>0$ and for any $\nu\in\mathbb{N}$ distinct points $\vec x_\nu,\vec y_\nu\in \tilde U_\delta$ such that $\abs{\vec x_\nu-\vec y_\nu}\to0$ as $\nu\to\infty$ and \begin{gather}\label{whitney_ext_thm_label_6}\abs{\partial^{\alpha_{0}}
F(\vec
x_\nu)-\sum\limits_{\abs{\gamma}\leq m-\abs{\alpha_{o}}}\frac{1}{\gamma!}\partial^{\alpha_{0}+\gamma}
F(\vec y_\nu)\cdot(\vec x_\nu-\vec y_\nu)^{\gamma}}\geq \eta\cdot\abs{\vec x_\nu-\vec y_\nu}^{m-\abs{\alpha_{0}}} \text{ for all }\nu\in\mathbb{N}.\end{gather}Since  $\tilde U_\delta$ is a compact that is contained in the open set $U_\delta$, by passing to a subsequence we may assume $\vec x_\nu,\vec y_\nu\to \vec z\in U_\delta$ as $\nu\to \infty$, and that moreover there exists a closed ball $B$ centered at $\vec z$ such that $B\subset U_\delta$ and $\vec x_\nu,\vec y_\nu\in B$ for any $\nu\in \mathbb{N}$. Note that $F$, and
all of its derivatives up to order $m$, are uniformly continuous on the closed ball $B$. Let $\omega(\cdot)$ be the modulus of continuity of the $m^{\text{th}}$ derivatives of $F$ on $B$. Then, \begin{gather}\label{whitney_ext_thm_label_7}\omega(t)\to0\text{ as }t\to0,\end{gather}and by Taylor's Theorem there exists a constant $\tilde C(m,n)>0$ depending only on $m$ and $n$ such that \begin{multline}\label{whitney_ext_thm_label_8}\abs{\partial^\alpha
F(\vec
x)-\sum\limits_{\abs{\gamma}\leq m-\abs{\alpha}}\frac{1}{\gamma!}\partial^{\alpha+\gamma}
F(\vec y)\cdot(\vec x-\vec y)^{\gamma}}\leq \tilde C(m,n)\cdot \omega(\abs{\vec x-\vec y})\cdot \abs{\vec x-\vec y}^{m-\abs{\alpha}}
\\ \text { for all }\vec x,\vec y\in B 
\text{ distinct, and all }\abs{\alpha}\leq m.\end{multline} Applying (\ref{whitney_ext_thm_label_7}) and (\ref{whitney_ext_thm_label_8}) with $\vec x=\vec x_\nu$ and $\vec y=\vec y_\nu$ we get \begin{gather}\label{whitney_ext_thm_label_9}\abs{\partial^{\alpha_{0}}
F(\vec
x_\nu)-\sum\limits_{\abs{\gamma}\leq m-\abs{\alpha_{o}}}\frac{1}{\gamma!}\partial^{\alpha_{0}+\gamma}
F(\vec y_\nu)\cdot(\vec x_\nu-\vec y_\nu)^{\gamma}}=o(\abs{\vec x_\nu-\vec y_\nu}^{m-\abs{\alpha_{0}}}) \text{ as }\nu\to\infty .\end{gather} Now, (\ref{whitney_ext_thm_label_9}) clearly contradicts (\ref{whitney_ext_thm_label_6}), and so we proved that (\ref{whitney_ext_thm_label_5}) holds. 

Set $P^{\vec x}=J_{\vec x}F$ (the $m^{th}$-jet of $F$ about $\vec x$) for any $\vec x\in\tilde U_\delta$. Now (\ref{whitney_ext_thm_label_1}), (\ref{whitney_ext_thm_label_2}) and (\ref{whitney_ext_thm_label_5}) tell us that \begin{gather}\label{whitney_ext_thm_label_1_prime}\abs{(\partial^\alpha
P^{\vec x})(\vec x)}\leq M\text { for all
}\vec x\in \tilde U_\delta\text{ and all }\abs{\alpha}\leq
m;\end{gather} \begin{multline}\label{whitney_ext_thm_label_2_prime}\abs{(\partial^\alpha P^{\vec x})(\vec
x)-\sum\limits_{\abs{\gamma}\leq m-\abs{\alpha}}\frac{1}{\gamma!}[(\partial^{\alpha+\gamma}
P^{\vec y})(\vec y)]\cdot(\vec x-\vec y)^{\gamma}}\leq M\abs{\vec x-\vec
y}^{m-\abs{\alpha}} \\ \text { for all }\vec x,\vec y\in \tilde U_\delta
\text{ distinct, and for all }\abs{\alpha}\leq m;\end{multline}and \begin{multline}\label{whitney_ext_thm_label_5_prime}\abs{(\partial^\alpha
P^{\vec x})(\vec
x)-\sum\limits_{\abs{\gamma}\leq m-\abs{\alpha}}\frac{1}{\gamma!}[(\partial^{\alpha+\gamma}
P^{\vec y})(\vec y)]\cdot(\vec x-\vec y)^{\gamma}}=o(\abs{\vec x-\vec y}^{m-\abs{\alpha}})
\\ \text { as }\abs{\vec x-\vec y}\to 0,\text{ subject to }\vec x,\vec y\in
\tilde U_\delta
\text{ distinct, for all }\abs{\alpha}\leq m.\end{multline}The above (\ref{whitney_ext_thm_label_1_prime})-(\ref{whitney_ext_thm_label_5_prime}) are the hypothesis of the classical Whitney's Extension Theorem (see \cite{W} and \cite[Theorem 2.3]{FI}), and so there exists a constant $C(m,n)>0$ depending only on $m$ and $n$ and $\tilde F\in C^m(\mathbb{R}^n)$ such that \begin{gather}\label{whitney_ext_thm_label_3_prime}\abs{\partial^\alpha\tilde
F(\vec x)}\leq C(m,n)\cdot M\text { for all
}\vec x\in \mathbb{R}^n\text{ and all }\abs{\alpha}\leq
m;\end{gather} and \begin{gather}\label{whitney_ext_thm_label_4_prime}J_{\vec x}(\tilde F)=P^{\vec x}(=J_{\vec x}F)\text { for all
}\vec x\in \tilde U_\delta.\end{gather} In particular (\ref{whitney_ext_thm_label_4_prime}) implies that \begin{gather}\label{whitney_ext_thm_label_4_prime_prime}\tilde F(\vec x)=F(\vec
x)\text { for all
}\vec x\in \tilde U_\delta,\end{gather} which together with (\ref{whitney_ext_thm_label_3_prime}) implies that (\ref{whitney_ext_thm_label_3}) and (\ref{whitney_ext_thm_label_4}) indeed hold.\end{proof}

\begin{lemma}\label{lemma_euqiv_def_of_neg_fun}Definitions \ref{def_negligible_functions} and \ref{def_whitney_negligible_functions} coincide(i.e., let
$U\subset\mathbb{R}^n$ be open, let $\Omega\subset S^{n-1}$ and let $F\in C^m(U)$. Then, $F$ is negligible for $\Omega$ if and only if $F$ is Whitney-negligible for $\Omega$).\end{lemma}

\begin{proof}If $\Omega=\emptyset$ then one readily sees that any $F\in C^m(U)$ is both negligible for $\Omega$ and Whitney-negligible
for $\Omega$, so we only need to show the equivalence of Definitions \ref{def_negligible_functions}
and \ref{def_whitney_negligible_functions} for $\Omega\neq\emptyset$. Let $F$ be Whitney-negligible for $\Omega\neq\emptyset$ and fix $\epsilon>0$. Let $\delta,r$ and $F_\epsilon$ be such that (\ref{whitney_negligible_fun_condition_1}) and (\ref{whitney_negligible_fun_condition_2}) hold. We immediately have from (\ref{whitney_negligible_fun_condition_1})
and (\ref{whitney_negligible_fun_condition_2}) that
\begin{gather}\label{negligible_equiv_def_lemma_label_1}\abs{\partial^\alpha
F(\vec x)}\leq\epsilon\abs{\vec x}^{m-\abs{\alpha}}\text { for all
}\vec x\in\Gamma(\Omega,\delta,r)\text{ and all }\abs{\alpha}\leq
m.\end{gather}  Moreover, (\ref{whitney_negligible_fun_condition_1}) and (\ref{whitney_negligible_fun_condition_2}) together with Taylor's Theorem imply that for some constant $C_1>0$ (that depends only on $m$ and $n$) we have \begin{multline}\label{negligible_equiv_def_lemma_label_2}\abs{\partial^\alpha F(\vec
x)-\sum\limits_{\abs{\gamma}\leq m-\abs{\alpha}}\frac{1}{\gamma!}\partial^{\alpha+\gamma}
F(\vec y)\cdot(\vec x-\vec y)^{\gamma}}\leq C_1\epsilon\abs{\vec x-\vec
y}^{m-\abs{\alpha}} \\ \text { for all }\vec x,\vec y\in\Gamma(\Omega,\delta,r)
\text{ distinct, and all }\abs{\alpha}\leq m.\end{multline}We conclude that $F$ is negligible for $U$.

\

Let $F$ be negligible for $\Omega$ and fix $\epsilon>0$. Let $\delta$ and $r$ be such that (\ref{negligible_fun_condition_0}), (\ref{negligible_fun_condition_a}) and (\ref{negligible_fun_condition_b})
hold and without loss of generality assume $r<1$. We recall that $$\Gamma(\Omega,\delta,r)=\{\vec
x\in\mathbb{R}^n:0<\abs{\vec x}<r,\text{dist}(\frac{\vec x}{\abs{\vec
x}},\Omega)<\delta\},$$ and for any $0<\rho\leq r$ we set $$E(\rho):=\{\vec
x\in\mathbb{R}^n:\frac{\rho}{10}<\abs{\vec x}<\rho,\text{dist}(\frac{\vec x}{\abs{\vec
x}},\Omega)<\delta\}.$$ We then have from (\ref{negligible_fun_condition_a})
and (\ref{negligible_fun_condition_b}) that for some constant $C_2>0$ that depends only on $m$ and $n$
\begin{gather}\label{negligible_equiv_def_lemma_label_3}\abs{\partial^\alpha
F(\vec x)}\leq C_{2}\epsilon\rho^{m-\abs{\alpha}}\text { for all
}\vec x\in E(\rho)\text{ and all }\abs{\alpha}\leq
m;\end{gather}
\begin{multline}\label{negligible_equiv_def_lemma_label_4}\abs{\partial^\alpha F(\vec
x)-\sum\limits_{\abs{\gamma}\leq m-\abs{\alpha}}\frac{1}{\gamma!}\partial^{\alpha+\gamma}
F(\vec y)\cdot(\vec x-\vec y)^{\gamma}}\leq C_{2} \epsilon\abs{\vec x-\vec
y}^{m-\abs{\alpha}} \\ \text { for all }\vec x,\vec y\in E(\rho)\text{ distinct, and all }\abs{\alpha}\leq m.\end{multline}
Define a function $G\in C^m(\rho^{-1}E(\rho))$ by $G(\vec x):= F(\rho \vec x)$. From (\ref{negligible_equiv_def_lemma_label_3}) and (\ref{negligible_equiv_def_lemma_label_4}) we get that 
\begin{gather}\label{negligible_equiv_def_lemma_label_5}\abs{\partial^\alpha
G(\vec x)}\leq C_{2}\epsilon\rho^{m}\text { for all
}\vec x\in \rho^{-1} E(\rho)\text{ and all }\abs{\alpha}\leq
m;\end{gather}
\begin{multline}\label{negligible_equiv_def_lemma_label_6}\abs{\partial^\alpha
G(\vec
x)-\sum\limits_{\abs{\gamma}\leq m-\abs{\alpha}}\frac{1}{\gamma!}\partial^{\alpha+\gamma}
G(\vec y)\cdot(\vec x-\vec y)^{\gamma}}\leq C_{2}\epsilon\rho^{m}\abs{\vec x-\vec
y}^{m-\abs{\alpha}} \\ \text { for all }\vec x,\vec y\in \rho^{-1} E(\rho)\text{ distinct,
and all }\abs{\alpha}\leq m.\end{multline}
Define $$\hat E(\rho):=\{\vec
x\in\mathbb{R}^n:\frac{\rho}{2}<\abs{\vec x}<\frac{2\rho}{3},\text{dist}(\frac{\vec
x}{\abs{\vec
x}},\Omega)<\frac{\delta}{2}\}.$$Applying Lemma \ref{new_cor_from_whitney} we find that there exists $G_\rho \in C^m(\mathbb{R}^n)$ such that \begin{gather}\label{negligible_equiv_def_lemma_label_7}\abs{\partial^\alpha
G_\rho (\vec x)}\leq C(m,n)\cdot C_{2} \epsilon\rho^{m}\text { for all
}\vec x\in \mathbb{R}^n\text{ and all }\abs{\alpha}\leq
m;\end{gather}
\begin{gather}\label{negligible_equiv_def_lemma_label_8}G_\rho (\vec x)=G(\vec
x)\text { for all
}\vec x\in  \rho^{-1} \hat E(\rho).\end{gather}Define a function $F_\rho \in C^m(\mathbb{R}^n)$ by $F_\rho (\vec x):= G_\rho (\rho^{-1} \vec
x)$. From (\ref{negligible_equiv_def_lemma_label_7}) and (\ref{negligible_equiv_def_lemma_label_8})
we get that \begin{gather}\label{negligible_equiv_def_lemma_label_9}\abs{\partial^\alpha
F_\rho(\vec x)}\leq C(m,n)\cdot C_{2}\epsilon\rho^{m-\abs{\alpha}}\text { for all
}\vec x\in \mathbb{R}^n\text{ and all }\abs{\alpha}\leq
m;\end{gather}
\begin{gather}\label{negligible_equiv_def_lemma_label_10}F_\rho(\vec x)=F(\vec
x)\text { for all
}\vec x\in \hat E(\rho).\end{gather}For $k\in \mathbb{N}\cup\{0\}$ set $\rho_k:=(\frac{9}{10})^{k}r$ and let $\theta_k\in C^m(\mathbb{R}^n)$ be such that the following hold: 
\begin{gather}\label{negligible_equiv_def_lemma_label_11}\sum\limits_{k=0}^{\infty}\theta_k(\vec x)=1\text { for all
}0<\abs{\vec x}\leq\frac{r}{10};\end{gather}
\begin{gather}\label{negligible_equiv_def_lemma_label_12}\text{supp}\theta_k\subset\{\frac{\rho_k}{2}\leq \abs{\vec x}\leq \frac{2\rho_k}{3}\};\end{gather}
\begin{gather}\label{negligible_equiv_def_lemma_label_13}\abs{\partial^\alpha
\theta_k(\vec x)}\leq C_3\cdot \rho_k^{-\abs{\alpha}}\text { for
all
}\vec x\in \mathbb{R}^n\text{ and all }\abs{\alpha}\leq
m,\end{gather}where $C_3$ is some constant depending on $m$ and $n$ only. Define $\tilde F\in C^m(\mathbb{R}^n\setminus\{\vec 0\})$ by $$\tilde F(\vec x):=\sum\limits_{k=0}^{\infty}\theta_k(\vec
x)\cdot F_{\rho_k}(\vec
x).$$ We have from (\ref{negligible_equiv_def_lemma_label_9}), (\ref{negligible_equiv_def_lemma_label_12}) and (\ref{negligible_equiv_def_lemma_label_13}) that for some constant $C_4>0$ that
depends only on $m$ and $n$ \begin{gather}\label{negligible_equiv_def_lemma_label_14}\abs{\partial^\alpha
\tilde F(\vec x)}\leq C_{4}\epsilon\abs{\vec x}^{m-\abs{\alpha}}\text {
for all
}\vec x\in \mathbb{R}^n\setminus\{\vec 0\}\text{ and all }\abs{\alpha}\leq
m,\end{gather}and from (\ref{negligible_equiv_def_lemma_label_10}), (\ref{negligible_equiv_def_lemma_label_11}) and (\ref{negligible_equiv_def_lemma_label_12}) we have \begin{gather}\label{negligible_equiv_def_lemma_label_15} \tilde F(\vec x)=F(\vec
x)\text { for all
}\vec x\in \Gamma(\Omega,\frac{\delta}{2},\frac{r}{10}).\end{gather}We conclude that
from (\ref{negligible_equiv_def_lemma_label_14}) and (\ref{negligible_equiv_def_lemma_label_15}) that $F$ is Whitney-negligible for $U$.\end{proof}

\begin{example}\label{amazing_exaple_for_negligible_part1}Set $n=3$, $m=2$
and $(x,y,z)$ a standard coordinate system on $\mathbb{R}^3$. Let $U=\mathbb{R}^3\setminus\{z=0\}$, $\Omega=\{(0,0,\pm1)\}\subset
S^2$ and let $F(x,y,z)=\frac{y^3}{z}\in C^2(U)$.
Then, $F$ is negligible for $\Omega\cap D(\omega,10^{-3})$, for any $\omega\in\Omega$.

Indeed, fixing $\omega\in\Omega$ we have $\Omega\cap D(\omega,10^{-3})=\{\omega\}$. Given $\epsilon>0$ (and we are allowed to assume $\epsilon<1$) one has to find $\delta,r>0$ such that (\ref{negligible_fun_condition_0}), (\ref{negligible_fun_condition_a}) and (\ref{negligible_fun_condition_b}) hold, with $\Omega$ being replaced by $\{\omega\}$. We claim that taking  $\delta=10^{-100}\cdot \epsilon^{10}$ and $r=1$ we satisfy (\ref{negligible_fun_condition_0}),
(\ref{negligible_fun_condition_a}) and (\ref{negligible_fun_condition_b}): indeed, (\ref{negligible_fun_condition_0}) holds trivially. In order to see  (\ref{negligible_fun_condition_a}) we note that on $\Gamma(\omega,10^{-100}\cdot \epsilon^{10},1)$ we have $\abs{y}<\abs{\frac{\epsilon }{10}z}$. So on this cone we have: 
$$\abs{F^{(0,0,0)}(x,y,z)}=\abs{\frac{y^3}{z}}<\abs{\frac{\epsilon^3
z^2}{1000}}<\epsilon (x^2+y^2+z^2)=\epsilon\abs{\vec x}^2;$$
$$\abs{F^{(0,1,0)}(x,y,z)}=\abs{\frac{3y^2}{z}}<\abs{\frac{3\epsilon^2
z}{100}}<\epsilon (x^2+y^2+z^2)^{1/2}=\epsilon\abs{\vec x};$$
$$\abs{F^{(0,0,1)}(x,y,z)}=\abs{\frac{-y^3}{z^{2}}}<\abs{\frac{\epsilon^3
z}{1000}}<\epsilon (x^2+y^2+z^2)^{1/2}=\epsilon\abs{\vec x};$$
$$\abs{F^{(0,2,0)}(x,y,z)}=\abs{\frac{6y}{z}}<\abs{\frac{6}{10}\epsilon}<\epsilon;\text{ }\abs{F^{(0,0,2)}(x,y,z)}=\abs{\frac{2y^3}{z^{3}}}<\abs{\frac{2\epsilon^3
}{1000}}<\epsilon;$$
$$\abs{F^{(0,1,1)}(x,y,z)}=\abs{\frac{-3y^2}{z^{2}}}<\abs{\frac{3\epsilon^2
}{100}}<\epsilon.$$
All the other partial derivatives of $F$ are identically zero, so we showed that indeed (\ref{negligible_fun_condition_a}) holds. Finally, as $\Gamma(\omega,10^{-100}\cdot
\epsilon^{10},1)$ is convex, Taylor's Theorem together with (\ref{negligible_fun_condition_a}) implies (\ref{negligible_fun_condition_b}).
\end{example}

\begin{lemma}[patching negligible functions]\label{patching_negligible_functions_lemma}Let
$U\subset\mathbb{R}^n$ be open and let $\Omega\subset S^{n-1}$. Let $\delta_1\dots,\delta_K>0$, let $\omega_1,\dots,\omega_K\in S^{n-1}$ and define $\Omega_k:=\Omega\cap D(\omega_{k},\delta_k)$ for any
$1\leq k\leq K$. Suppose that for any $1\leq k\leq K$ we are given $F_k\in C^m(U)$, such that $F_k$ is Whitney-negligible (or equivalently negligible) for $\Omega_k$. Suppose that moreover we are given $\theta_1,\dots,\theta_K\in C^m(\mathbb{R}^n\setminus\{\vec 0\})$ and a constant $\hat C>0$ such that for any $1\leq k\leq K$ we have \begin{gather}\label{patching_neg_lemma_label_1}\text{supp}\theta_k\subset\{\vec x\in\mathbb{R}^n:\abs{\frac{\vec x}{\abs{\vec
x}}-\omega_k}\leq\frac{2}{3}\delta_k\};\end{gather}and
\begin{gather}\label{patching_neg_lemma_label_2}\abs{\partial^\alpha\theta_k(\vec x)}\leq\hat C\abs{\vec x}^{-\abs{\alpha}}\text { for all
}\vec x\in \mathbb{R}^n\setminus\{\vec 0\}\text{ and all }\abs{\alpha}\leq
m.\end{gather}Then, $F(\vec x)=\sum\limits_{k=1}^{K}\theta_k(\vec x)\cdot F_k(\vec x)$ is Whitney-negligible (or equivalently negligible) for $\Omega$.

\end{lemma}
\begin{proof}Fix $\epsilon>0$ and for each $1\leq k\leq K$ let $r_k,\tilde\delta_k>0$ and $\tilde F_{\epsilon,k}\in C^m(\mathbb{R}^n\setminus\{\vec0\})$ be such that (\ref{whitney_negligible_fun_condition_0}), (\ref{whitney_negligible_fun_condition_1}) and (\ref{whitney_negligible_fun_condition_2}) hold for $\Omega_k$. So we have \begin{gather}\label{patching_neg_lemma_label_3}\abs{\partial^\alpha
\tilde F_{\epsilon,k}(\vec x)}\leq\epsilon\abs{\vec x}^{m-\abs{\alpha}}\text { for all
}\vec x\in\mathbb{R}^n\setminus\{\vec0\}\text{ and all }\abs{\alpha}\leq
m;\end{gather}
\begin{gather}\label{patching_neg_lemma_label_4}\tilde F_{\epsilon,k}(\vec x)=F_{k}(\vec
x)\text { for all
}\vec x\in\Gamma(\Omega_{k},\tilde \delta_k,r_k).\end{gather}
Since (\ref{whitney_negligible_fun_condition_0}), (\ref{whitney_negligible_fun_condition_2}) and (\ref{patching_neg_lemma_label_4}) are preserved under replacing $\delta$ and $\tilde \delta_k$ by smaller numbers, we may assume without loss of generality that $\tilde \delta_k\leq\delta_k$ for each $1\leq k \leq K$. Set $$\tilde F_{\epsilon}(\vec x):=\sum\limits_{k=1}^{K}\theta_k(\vec x)\cdot \tilde F_{\epsilon,k}(\vec x).$$ By (\ref{patching_neg_lemma_label_2}) and (\ref{patching_neg_lemma_label_3}) for some constant $C_1>0$ (that may be dependent on $K$ as well, in addition to $m$ and $n$) we have \begin{gather}\label{patching_neg_lemma_label_5}\abs{\partial^\alpha
\tilde F_{\epsilon}(\vec x)}\leq C_{1}\hat C\epsilon\abs{\vec x}^{m-\abs{\alpha}}\text
{ for all
}\vec x\in\mathbb{R}^n\setminus\{\vec0\}\text{ and all }\abs{\alpha}\leq
m.\end{gather}Now set $r:=10^{-10}\min\{r_1,\dots ,r_K\}$, $\delta:=10^{-9}\min\{\tilde\delta_1,\dots,\tilde\delta_K\}$, and let $\vec x\in\Gamma(\Omega,\delta,r)$. In particular we have $0<\abs{\vec x}<r\leq r_k$ for any $1\leq k\leq K$. Suppose that moreover $\vec x\in \text{supp}\theta_k$ for some $1\leq k\leq K$. We then have by (\ref{patching_neg_lemma_label_1}) that $$\abs{\frac{\vec x}{\abs{\vec x}}-\omega_k}\leq\frac{2}{3}\delta_k,$$and moreover since $\vec x\in\Gamma(\Omega,\delta,r)$ we also have $$\abs{\frac{\vec x}{\abs{\vec x}}-\omega''}\leq10^{-9}\tilde\delta_k\leq10^{-9}\delta_k\text{ for some }\omega''\in\Omega.$$ We conclude that $$\abs{\omega''-\omega_k}\leq\frac{2}{3}\delta_k+10^{-9}\delta_k<\delta_k\text{ for some }\omega''\in\Omega,$$hence $\omega''\in \Omega_k$, and so $\text{dist}(\frac{\vec x}{\abs{\vec x}},\Omega_k)<\tilde \delta_k$. Recall that also $0<\abs{\vec x}\leq r_k$ so $\vec x\in\Gamma(\Omega_k,\tilde \delta_k,r_k)$. We conclude from (\ref{patching_neg_lemma_label_4}) that \begin{gather}\label{patching_neg_lemma_label_6}\tilde
F_{\epsilon,k}(\vec
x)=F_k(\vec
x)\text { for all
}\vec x\in\Gamma(\Omega,\delta,r)\cap\text{supp}\theta_k.\end{gather}Consequently we have $$\tilde F_{\epsilon}(\vec x)=\sum\limits_{k=1}^{K}\theta_k(\vec x)\cdot\tilde
F_{\epsilon,k}(\vec
x)=\sum\limits_{k=1}^{K}\theta_k(\vec x)\cdot
F_k(\vec x)=F(\vec x)\text { for all
}\vec x\in\Gamma(\Omega,\delta,r),$$which together with (\ref{patching_neg_lemma_label_5}) proves that $F$ is Whitney-negligible for $\Omega$.\end{proof}

\subsection{Strong directional implication and strong implication}

\begin{definition}[strong directional implication]\label{implied_polynomial_definition}Let $I\lhd\mathcal{P}_0^{m}(\mathbb{R}^n)$
be an ideal and $p\in\mathcal{P}_0^{m}(\mathbb{R}^n)$ be some polynomial. We
say that\textit{ $I$ strongly implies $p$ in the direction $\omega\in S^{n-1}$
}if there exist $\delta_\omega>0$, $r_\omega>0$, polynomials $Q_1,\dots ,Q_L\in I$,
functions $ S_1,\dots ,S_L\in C^m(\Gamma(\omega,\delta_\omega,r_\omega))$,  positive
constants $C_1,\dots ,C_L>0$ and a function $F\in C^m(\Gamma(\omega,\delta_\omega,r_\omega))$
 such that the following hold:
\begin{gather}\label{first_cond_of_implied_poly_in_direction}F\text{
is negligible for }Allow(I)\cap D(\omega,\delta_\omega);\end{gather}
\begin{gather}\label{second_cond_of_implied_poly_in_direction}\abs{\partial
^\alpha S_{l}(\vec x)}\leq C_{l}\abs{\vec x}^{-\abs{\alpha}}\text{ for all
}\abs{\alpha}\leq
m, \text{ all }1\leq l \leq L\text{ and all }\vec x\in\Gamma(\omega,\delta_\omega,r_\omega);\end{gather}
\begin{gather}\label{third_cond_of_implied_poly_in_direction}p(\vec x)=\sum\limits_{l=1}^{L}S_l(\vec
x)\cdot Q_l(\vec x)+F(\vec x)\text { for all }\vec x\in\Gamma(\omega,\delta_\omega,r_\omega).\end{gather}\end{definition}

\begin{remark}\label{remark_strong_implication_holds_trivially_in_forb_directions}If $\omega\notin Allow(I)$ then $I$ always strongly implies $p$ in the direction $\omega$, for any $p\in\mathcal{P}_0^{m}(\mathbb{R}^n)$. Indeed, recall that $Allow(I)$ is closed (see Definition \ref{definition_allowed_and_forbidden_directions}) and fix some $\delta_\omega>0$ such that $Allow(I)\cap D(\omega,\delta_\omega)=\emptyset$.
 Now (\ref{first_cond_of_implied_poly_in_direction})--(\ref{third_cond_of_implied_poly_in_direction}) hold with $F=p$, $L=C_{1}=r_{\omega}=1$, $Q_1=S_1=0$. \end{remark}

\begin{definition}[strong implication]\label{strong_implied_polynomial_definition}Let
$I\lhd\mathcal{P}_0^{m}(\mathbb{R}^n)$
be an ideal and $p\in\mathcal{P}_0^{m}(\mathbb{R}^n)$ be some polynomial. We
say that\textit{ $I$ strongly implies  $p$} if there exist $r_0>0$, polynomials $Q_1,\dots ,Q_L\in
I$,
functions $ S_1,\dots ,S_L\in C^m(B^\times(r_0))$,
 positive
constants $C_1,\dots ,C_L>0$ and a function $F\in C^m(B^\times(r_0))$
 such that the following hold:
\begin{gather}\label{first_cond_of_strong_implied_poly_in_direction}F\text{
is negligible for }Allow(I);\end{gather}
\begin{gather}\label{second_cond_of_strong_implied_poly_in_direction}\abs{\partial
^\alpha S_{l}(\vec x)}\leq C_{l}\abs{\vec x}^{-\abs{\alpha}}\text{ for all
}\abs{\alpha}\leq
m, \text{ all }1\leq l \leq L\text{ and all }\vec x\in B^\times(r_0);\end{gather}
\begin{gather}\label{third_cond_of_strong_implied_poly_in_direction}p(\vec x)=\sum\limits_{l=1}^{L}S_l(\vec
x)\cdot Q_l(\vec x)+F(\vec x)\text { for all }\vec x\in B^\times(r_0).\end{gather}\end{definition}

\begin{remark}Let  
$I\lhd\mathcal{P}_0^{m}(\mathbb{R}^n)$ be an ideal
 and let  $p\in\mathcal{P}_0^{m}(\mathbb{R}^n)$ be a jet. Clearly, if $I$ strongly implies $p$ then $I$
strongly implies $p$ in the direction $\omega$ for any $\omega\in Allow(I)$. The following Lemma \ref{strong_directional_imply_strog_lemma} shows that the converse also holds.\end{remark} 

\begin{lemma}\label{strong_directional_imply_strog_lemma}Let
$I\lhd\mathcal{P}_0^{m}(\mathbb{R}^n)$
be an ideal and $p\in\mathcal{P}_0^{m}(\mathbb{R}^n)$ be some polynomial. If $I$ strongly implies $p$ in the direction $\omega$ for any $\omega\in Allow(I)$, then $I$ strongly implies $p$. \end{lemma}
\begin{proof}By Remark \ref{remark_strong_implication_holds_trivially_in_forb_directions} and our assumption, we have that $I$ strongly implies $p$ in any direction $\omega\in S^{n-1}$. Fix $Q_1,\dots, Q_L$ a basis of $I$ (as a vector space). Then, for  each $\omega\in S^{n-1}$ there exist $\delta_\omega>0$, $r_\omega>0$, 
functions $ S^\omega_1,\dots S^\omega_L\in C^m(\Gamma(\omega,\delta_\omega,r_\omega))$,
 positive
constants $C^\omega_1,\dots ,C^\omega_L>0$ and a function $F^\omega\in C^m(\Gamma(\omega,\delta_\omega,r_\omega))$
 such that the following hold:
\begin{gather}\label{strong_directional_imply_strog_lemma_label_1}F^\omega\text{
is negligible for }Allow(I)\cap D(\omega,\delta_\omega);\end{gather}
\begin{gather}\label{strong_directional_imply_strog_lemma_label_1_prime}\abs{\partial
^\alpha S^\omega_{l}(\vec x)}\leq C^\omega_{l}\abs{\vec x}^{-\abs{\alpha}}\text{ for all
}\abs{\alpha}\leq
m, \text{ all }1\leq l \leq L\text{ and all }\vec x\in\Gamma(\omega,\delta_\omega,r_\omega);\end{gather}
\begin{gather}\label{strong_directional_imply_strog_lemma_label_3}p(\vec x)=\sum\limits_{l=1}^{L}S^\omega_l(\vec
x)\cdot Q_l(\vec x)+F^\omega(\vec x)\text { for all }\vec x\in\Gamma(\omega,\delta_\omega,r_\omega).\end{gather}
By compactness of $S^{n-1}$ there exists finitely many $\omega_1,\dots,\omega_K$ such that $S^{n-1}=\bigcup\limits_{k=1}^K D(\omega_k,\frac{\delta_{\omega_k}}{100})$. Fix $\tilde\theta_1,\dots,\tilde\theta_K\in C^\infty(S^{n-1})$ such that $\sum\limits_{k=1}^K\tilde\theta_k(\omega)=1$ for any $\omega\in S^{n-1}$ and $\text{supp}\tilde \theta_k\subset D(\omega_k,\frac{\delta_{\omega_k}}{50})$ for all $1\leq k\leq K$. Define $\theta_k(\vec x):=\tilde\theta_k(\frac{\vec x}{\abs{\vec x}})$ for all $\vec x\in\mathbb{R}^n\setminus\{\vec 0\}$ and
$1\leq k\leq K$ we get that for some constant $\hat C>0$ (depending on $m,n$ and $\theta_1,\dots,\theta_K$) we have  
\begin{gather}\label{strong_directional_imply_strog_lemma_label_4}\theta_1,\dots,\theta_K\in C^\infty(\mathbb{R}^n\setminus\{\vec 0\});\end{gather}
\begin{gather}\label{strong_directional_imply_strog_lemma_label_5}\sum\limits_{k=1}^K\theta_k(\vec x)=1\text{ for all }\vec x\in\mathbb{R}^n\setminus\{\vec 0\};\end{gather}
\begin{gather}\label{strong_directional_imply_strog_lemma_label_6}\theta_k(\vec
x)=0\text{ if }\abs{\frac{\vec x}{\abs{\vec x}}-\omega_k}>\frac{2}{3}\delta_{\omega_k}\text{ for all }1\leq k\leq K;\end{gather}
\begin{gather}\label{strong_directional_imply_strog_lemma_label_7}\abs{\partial^\alpha\theta_k(\vec
x)}\leq\hat C\abs{\vec x}^{-\abs{\alpha}}\text { for all
}\vec x\in \mathbb{R}^n\setminus\{\vec 0\}\text{ and all }\abs{\alpha}\leq
m.\end{gather} Set $\hat r=\min\{r_{\omega_1},\dots,r_{\omega_k}\}$. Thanks to (\ref{strong_directional_imply_strog_lemma_label_6}) we can define $F,S_1,S_2,\dots, S_L\in C^m(B^\times(\hat r))$ by 
\begin{gather}\label{strong_directional_imply_strog_lemma_label_8}F(\vec x):=\sum\limits_{k=1}^K\theta_k(\vec
x)F^{\omega_k}(\vec x)\text{ and }S_{l}(\vec
x):=\sum\limits_{k=1}^K\theta_k(\vec
x)S^{\omega_k}_l(\vec x)\text { for all
}\vec x\in \mathbb{R}^n\setminus\{\vec 0\}.\end{gather}Thanks to (\ref{strong_directional_imply_strog_lemma_label_1}), (\ref{strong_directional_imply_strog_lemma_label_6}), (\ref{strong_directional_imply_strog_lemma_label_7}) and Lemma \ref{patching_negligible_functions_lemma} we get that
\begin{gather}\label{strong_directional_imply_strog_lemma_label_9}F\text{
is negligible for }Allow(I).\end{gather}Thanks to (\ref{strong_directional_imply_strog_lemma_label_1_prime}) and (\ref{strong_directional_imply_strog_lemma_label_7}) we get that for some constant $C'>0$ (depending on $m,n,\hat C$
and $\{C_l^{\omega_k}\}_{1\leq k\leq K,1\leq l\leq L}$) we have
\begin{gather}\label{strong_directional_imply_strog_lemma_label_10}\abs{\partial
^\alpha S_{l}(\vec x)}\leq C'\abs{\vec x}^{-\abs{\alpha}}\text{
for all
}\abs{\alpha}\leq
m, \text{ all }1\leq l \leq L\text{ and all }\vec x\in B^\times(\hat r).\end{gather}
Finally, from (\ref{strong_directional_imply_strog_lemma_label_3}), (\ref{strong_directional_imply_strog_lemma_label_6}), and (\ref{strong_directional_imply_strog_lemma_label_5}) we get that 
\begin{gather}\label{strong_directional_imply_strog_lemma_label_11}p(\vec
x)=\sum\limits_{l=1}^{L}S_l(\vec
x)\cdot Q_l(\vec x)+F(\vec x)\text { for all }\vec x\in B^\times(\hat r).\end{gather}
Combining (\ref{strong_directional_imply_strog_lemma_label_9}), (\ref{strong_directional_imply_strog_lemma_label_10}) and (\ref{strong_directional_imply_strog_lemma_label_11}) we proved that $I$ strongly implies $p$.\end{proof}

\begin{lemma}\label{strong_implication_imply_implication_lemma}Let
$I\lhd\mathcal{P}_0^{m}(\mathbb{R}^n)$
be an ideal and $p\in\mathcal{P}_0^{m}(\mathbb{R}^n)$ be some polynomial. If $I$
strongly implies $p$ (as in Definition \ref{strong_implied_polynomial_definition}),
then $I$ implies $p$ (as in Definition \ref{new_def_implied_jet}). \end{lemma}
\begin{proof}Assume $I$ strongly implies $p$ and denote $\Omega:=Allow(I)$ . Let $r_0>0$, $Q_1,\dots ,Q_L\in
I$, $ S_1,\dots ,S_L\in C^m(B^\times(r_0))$,
  $C_1,\dots ,C_L>0$ and $F\in C^m(B^\times(r_0))$
be such that (\ref{first_cond_of_strong_implied_poly_in_direction}), (\ref{second_cond_of_strong_implied_poly_in_direction}) and (\ref{third_cond_of_strong_implied_poly_in_direction}) hold. Setting $A:=\max\{C_1,\dots,C_L\}>0$ we have
\begin{gather}\label{strong_implication_imply_implication_lemma_label_1}F\text{
is negligible for }Allow(I);\end{gather}
\begin{gather}\label{strong_implication_imply_implication_lemma_label_2}\abs{\partial
^\alpha S_{l}(\vec x)}\leq A\abs{\vec x}^{-\abs{\alpha}}\text{ for all
}\abs{\alpha}\leq
m, \text{ all }1\leq l \leq L\text{ and all }\vec x\in B^\times(r_0);\end{gather}
\begin{gather}\label{strong_implication_imply_implication_lemma_label_3}p(\vec
x)=\sum\limits_{l=1}^{L}S_l(\vec
x)\cdot Q_l(\vec x)+F(\vec x)\text { for all }\vec x\in B^\times(r_0).\end{gather}
Fix $\epsilon>0$. By (\ref{strong_implication_imply_implication_lemma_label_1}), Definition \ref{def_whitney_negligible_functions} and Lemma \ref{lemma_euqiv_def_of_neg_fun} there exist $\delta,r>0$ (and without loss of generality $r<r_0$) and $F_\epsilon\in C^m(\mathbb{R}^n\setminus\{\vec 0\})$ such that
\begin{gather}\label{strong_implication_imply_implication_lemma_label_4}\abs{\partial^\alpha
F_\epsilon(\vec x)}\leq\epsilon\abs{\vec x}^{m-\abs{\alpha}}\text { for all
}\vec x\in\mathbb{R}^n\setminus\{\vec0\}\text{ and all }\abs{\alpha}\leq
m;\end{gather}
\begin{gather}\label{strong_implication_imply_implication_lemma_label_5}F_\epsilon(\vec x)=F(\vec
x)\text { for all
}\vec x\in\Gamma(\Omega,\delta,4r).\end{gather}
Let $0<\rho\leq r$. From (\ref{strong_implication_imply_implication_lemma_label_2}) and (\ref{strong_implication_imply_implication_lemma_label_4}) we have that for some constant $C_1>0$ (depending only on $m$ and $n$) the following hold 
\begin{gather}\label{strong_implication_imply_implication_lemma_label_6}\abs{\partial
^\alpha S_{l}(\vec x)}\leq C_{1}A\rho^{-\abs{\alpha}}\text { for all
}\vec x\in\text{Ann}_4(\rho)\text{ and all }\abs{\alpha}\leq
m;\end{gather}
\begin{gather}\label{strong_implication_imply_implication_lemma_label_7}\abs{\partial^\alpha
F_\epsilon(\vec x)}\leq C_{1}\epsilon\rho^{m-\abs{\alpha}}\text { for all
}\vec x\in\text{Ann}_4(\rho)\text{ and all }\abs{\alpha}\leq
m.\end{gather}
Finally, from (\ref{strong_implication_imply_implication_lemma_label_3}) and (\ref{strong_implication_imply_implication_lemma_label_5}) we immediately get that
\begin{multline}\label{strong_implication_imply_implication_lemma_label_8}p(\vec x)=F_\epsilon(\vec x)+S_1(\vec x)Q_1(\vec x)+S_2(\vec
x)Q_2(\vec x)+\dots+S_L(\vec
x)Q_L(\vec x) \\ \text {for all }\vec x\in\text{Ann}_2(\rho)
\text{ such that }\text{dist}(\frac{\vec x}{\abs{\vec x}},\Omega)<\delta.
\end{multline}
We thus showed that there exists a constant $A>0$ such that given $\epsilon>0$ there exists $\delta,r>0$ such that for any $0<\rho\leq r$ there exist functions $F_\epsilon,S_1,\dots,S_L\in C^m(\text{Ann}_4(\rho))$ such that (\ref{strong_implication_imply_implication_lemma_label_6}) (\ref{strong_implication_imply_implication_lemma_label_7}) and (\ref{strong_implication_imply_implication_lemma_label_8}) hold. That is, we showed that $I$ implies $p$.\end{proof}

\begin{remark}We do not know whether the converse of Lemma \ref{strong_implication_imply_implication_lemma} holds, i.e.: let
$I\lhd\mathcal{P}_0^{m}(\mathbb{R}^n)$
be an ideal and $p\in\mathcal{P}_0^{m}(\mathbb{R}^n)$ be some polynomial.  Assume that $I$ implies $p$ (as in Definition \ref{new_def_implied_jet}). Is it always true that  $I$ strongly implies $p$ (as in Definition \ref{strong_implied_polynomial_definition})?\end{remark}

\begin{corollary}\label{cor_strong_directional_implication_imply_implication}Let
$I\lhd\mathcal{P}_0^{m}(\mathbb{R}^n)$
be an ideal and $p\in\mathcal{P}_0^{m}(\mathbb{R}^n)$ be some polynomial. If $I$
strongly implies $p$ in the direction $\omega$ for any $\omega\in Allow(I)$ (as in Definition \ref{implied_polynomial_definition}),
then $I$ implies $p$ (as in Definition \ref{new_def_implied_jet}).\end{corollary}
\begin{proof}It follows immediately from Lemma \ref{strong_directional_imply_strog_lemma} and Lemma \ref{strong_implication_imply_implication_lemma}.\end{proof}

\begin{example}\label{amazing_exaple_for_negligible_part2}Set $n=3$, $m=2$
and $(x,y,z)$ a standard coordinate system on $\mathbb{R}^3$. Let $I\lhd\mathcal{P}_0^2(\mathbb{R}^3)$
be such that $x^2,y^2-xz\in I$. Then, $xy\in cl(I)$. In particular $\langle x^2,y^2-xz\rangle_2$
is not closed.

Indeed, by Corollary \ref{old_new_cor_on_how_to_calc_allow_with_inclusion}
we have $Allow(I)\subset\{(0,0,\pm1)\}$. By Example \ref{amazing_exaple_for_negligible_part1} the function $F=\frac{y^3}{z}$ is negligible for $\Omega\cap D(\omega,10^{-3})$, for any $\omega\in Allow(I)$. So by
Definition \ref{implied_polynomial_definition}, setting $L=1$, $S_1=-\frac{y}{z}$
and $Q_1=y^2-xz$ we have that $I$ strongly implies the jet $$xy=S_1\cdot Q_1+ F=(-\frac{y}{z})\cdot(y^2-xz)+\frac{y^3}{z}$$ in the direction
$\omega$, for any $\omega\in Allow(I)$. By Corollary \ref{cor_strong_directional_implication_imply_implication} $I$ implies $xy$, i.e., $xy\in cl(I)$. \end{example}

\section{An algorithm to calculate the closure of an ideal}\label{section_algorithm}

\subsection{Background on bundles}

We provide the necessary background on \textit{bundles}, introduced and studied
in (\cite{F1,FL1}). The notation in this section is slightly different
from the rest of this paper, as we adopt here the standard notation in the
literature when such bundles are used. Thus, we start by fixing notation.

\begin{notation}Fix natural numbers $m,n,\mathcal{D}\geq1$. We write $J_{\vec
x}(F)$ (or $J_{\vec x}^mF$) to denote the $m^{th}$ degree Taylor polynomial of a real
valued function $F\in C^m(U)$, when $U$ is an open neighborhood  of $\vec
x$. The vector space of all such polynomials is denoted by $\mathcal{P}$
(note that as a finite dimensional real vector space $\mathcal{P}$ can be
identified with $\mathbb{R}^{\mathcal{D}^*}$ for some $\mathcal{D}^*\in\mathbb{N}$,
so we can talk about semi-algebraic subsets of $\mathcal{P}$; see Footnote \ref{semialgebraic_footnote}). The ring of
$m$-jets (of real valued functions) at $\vec x\in\mathbb{R}^n$ is the vector
space $\mathcal{P}$, with multiplication $P\odot_{\vec x} Q=J_{\vec x}(PQ)$.
Then, for two $C^m$ functions around $\vec x$ we have $J_{\vec x}(FG)=J_{\vec
x}(F)\odot_{\vec x}J_{\vec x}(G)$. We write $\mathcal{R}_{\vec x}=(\mathcal{P},\odot_{\vec
x})$ to denote the ring of $m$-jets at $\vec x$.

The $m$-jet of a ($C^m$) $\mathbb{R}^{\mathcal{D}}$-valued function $\vec
F=(F_1,\dots, F_\mathcal{D})$ at $\vec x\in\mathbb{R}^n$ is defined to be
$J_{\vec x}^m(\vec F)=(J_{\vec x}^mF_1,\dots, J_{\vec x}^mF_\mathcal{D})\in\mathcal{P}^\mathcal{D}$.
The multiplication $Q\odot_{\vec x}(P_1,\dots,P_\mathcal{D}):=(Q\odot_{\vec x}P_1,\dots,Q\odot_{\vec
x}P_\mathcal{D})$ for $Q,P_1,\dots,P_\mathcal{D}\in\mathcal{P}$ turns $\mathcal{P}^\mathcal{D}$
into a $\mathcal{R}_{\vec x}$ module, which we denote by $\mathcal{R}_{\vec
x}^{\mathcal{D}}$. We will examine $\mathcal{R}_{\vec x}$ submodules of
$\mathcal{R}_{\vec x}^\mathcal{D}$, and we adopt the unusual convention that
$\{0\},\mathcal{R}_{\vec x}^\mathcal{D}$ and the empty set are all allowed
as submodules of $\mathcal{R}_{\vec x}^\mathcal{D}$.\end{notation} 

\textbf{Warning.} Before we proceed we stress that this section deals with
ideals
in $\mathcal{P}_0^{m}(\mathbb{R}^n)$
and how to calculate their closures. At no point in this section do we make
any assumption on an ideal $I\lhd \mathcal{P}_0^{m}(\mathbb{R}^n)$ being of the
form $I^m(E)$ for some closed subset $E\subset \mathbb{R}^n$ containing the
origin, nor do we show any ideal is of this form. The sets denoted $E$ below
play the role of the base space for our
bundles, and have nothing to do with closed subsets $E\subset\mathbb{R}^n$
containing the
origin that we studied in previous sections of this paper. We use the
same letter to stay in line with the standard notation in the literature.

\begin{definition}[bundles and sections]Let $E\subset\mathbb{R}^n$ be compact.
A \textit{bundle} (over $E$) is a family 
\begin{gather}\label{BB_label_1}\mathcal{H}=\{H_{\vec x}\}_{\vec x\in E}\text{
};\text{ for any }\vec x\in E, H_{\vec x}\text{ is a translate of an }\mathcal{R}_{\vec
x}\text{ submodule of }\mathcal{R}^\mathcal{D}_{\vec x}.\end{gather}
We call $H_{\vec x}$ the \textit{fiber} of $\mathcal{H}$ over $\vec x$. The
bundle (\ref{BB_label_1}) is called \textit{semi-algebraic} if the set $$\{(\vec
x,\vec P)\in\mathbb{R}^n\times\mathcal{P}^\mathcal{D}:\vec x\in E\text{ and
}\vec P\in H_{\vec {x}}\}\subset \mathbb{R}^n\times\mathcal{P}^\mathcal{D}$$
is semi-algebraic. When we say that we are \textit{given} (respectively \textit{compute})
a bundle, we mean that we are given (resp. compute) this set\footnote{For detailed discussions of the notion of "computing a set" see \cite[Section 2.5]{FL2} or more generally \cite{BaCR}.}.

A \textit{section} of the bundle (\ref{BB_label_1}) is a ($C^m$) $\mathbb{R}^{\mathcal{D}}$-valued
function $\vec F=(F_1,\dots, F_D)$ defined on $\mathbb{R}^n$, such that $J_{\vec
x}\vec F\in H_{\vec x}$ for any $\vec x\in E$. The 
$C^m$ norm of a section, denoted $\abs{\abs{\vec
F}}_{C^m}$, is defined by$$\abs{\abs{\vec
F}}_{C^m}:=\sup \big\{\abs{\partial^\alpha
F_i(\vec x)}\big\}_{1\leq i\leq \mathcal{D},\text{ }\abs{\alpha}\leq m,\text{
}\vec x\in \mathbb{R}^n}\in\mathbb{R}_{\geq0}\cup\{\infty\}.$$

For any $k\in\mathbb{N}$, the \textit{$k$-norm} of the bundle (\ref{BB_label_1}),
denoted $\abs{\abs{\mathcal{H}}}_k$, is infimum over all $\mathbb{R}\ni M>0$
for which the following holds:\begin{multline}\label{BB_label_2}\text{For
any }\vec x_1,\dots,\vec x_k\in E \text{ there exist }P_1\in H_{\vec x_1},\dots,P_k\in
H_{\vec x_k}\text{ such that} \\ \abs{\partial^\alpha P_i(\vec
x_i)}\leq M\text{ for all }1\leq i\leq
k \text{ and all }\abs{\alpha}\leq m\text{ }\text{ }\text{ }\text{ }\text{
}\text{ }\text{ }\text{ and, for }\vec x_i\neq \vec x_j, \\ \abs{\partial^\alpha(P_i-P_j)(\vec x_i)}\leq
M\abs{x_i-x_j}^{m-\abs{\alpha}}\text{ for all }\abs{\alpha}\leq m. \end{multline}
If no such $M$ exists we set $\abs{\abs{\mathcal{H}}}_k=\infty$.\end{definition}

Given a bundle $\mathcal{H}$, we would like to know whether it has a section,
and moreover if a section exists, we want to know how small can we take its
$C^m$ norm. By our convention, $\mathcal{H}$ may have empty fibers, in which
case clearly a section does not exist. To answer these questions \cite{F1,FL1} proved the existence and studied the properties of the \textit{stable
Glaeser refinement} of a bundle. We refer the reader to \cite{F1,FL1} for the definition and study of these stable refinements, and list here
only the properties we will use: given a bundle $\mathcal{H}=\{H_{\vec x}\}_{\vec
x\in E}$ the stable Glaeser refinement of $\mathcal{H}$ is another bundle
$\tilde {\mathcal{H}}=\{\tilde H_{\vec x}\}_{\vec x\in E}$. In particular
it satisfies: 

\begin{theorem}\label{theorem_on_bundles}There exist an integer constant
$k^\#\geq1$ and positive constants $c,C>0$, all three depending only on $m,n$
and $\mathcal{D}$, such that the following holds: let $\mathcal{H}=\{H_{\vec
x}\}_{\vec x\in E}\text{
}$ be a bundle, and let $\tilde {\mathcal{H}}=\{\tilde
H_{\vec x}\}_{\vec x\in E}$ be its stable Glaeser refinement. Then:
\begin{gather}\label{BB_label_3}\mathcal{H}\text{ has a section if and only
if }\tilde{\mathcal{H}}_{\vec x}\text{ is non empty for any }\vec x\in E;\end{gather}
\begin{gather}\label{BB_label_4}\text{if }\mathcal{H}\text{ has a section
then }c\abs{\abs{\mathcal{\tilde{H}}}}_{k^\#}\leq\inf\big\{\abs{\abs{\vec
F}}_{C^m}:\vec F\text{ is a section of }\mathcal{H}\big\}\leq C\abs{\abs{\mathcal{\tilde{H}}}}_{k^\#};\end{gather}
\begin{gather}\label{BB_label_5}\text{if }\mathcal{H}\text{ is semi-algebraic,
then so is }\tilde{\mathcal{H}},\text{ and moreover }\tilde{\mathcal{H}}\text{
can be computed if }\mathcal{H}\text{ is given}.\end{gather}\end{theorem}

The scalar case ($\mathcal{D}=1$) of (\ref{BB_label_3}) and (\ref{BB_label_4})
are proven in \cite{F1}. The general case ($\mathcal{D}\geq1$)
of (\ref{BB_label_3}) is proven in \cite{FL1} by reduction to the
scalar case. The same reduction can be used to prove the general case of
(\ref{BB_label_4}). The proof of (\ref{BB_label_5}) is a routine application
of standard properties of semi-algebraic sets\footnote{More precisely, the
stable Glaeser refinement is constructed by iterating finitely many times
the process of \textit{Glaeser refinement} (until this process stabilizes,
hence the term "stable"). Each iteration is defined by formulae that are
first order definable in the language of real fields with parameters from
$\mathbb{R}$, and so each iteration preserves the property of the bundle
being semi-algebraic. We refer the reader to \cite{F1} for further details on the
iterated refinement process, and to \cite{vdD} for exposition on the model theoretic
notion of definability in first order languages.}.
\begin{definition}The \textit{norm} of the bundle (\ref{BB_label_1}), denoted
$\abs{\abs{\mathcal{H}}}$, is defined to be the $k^{\#}$-norm of (\ref{BB_label_1}),
where $k^\#$ is the integer constant from Theorem \ref{theorem_on_bundles}
(formally this $k^\#$ is not unique, so we choose and fix one such $k^\#$).\end{definition}
Taylor's Theorem easily implies that there exists a constant $C'$, depending
only on $m,n$ and $\mathcal{D}$, such that $$\abs{\abs{\mathcal{H}}}\leq
C'\abs{\abs{\vec F}}_{C^m}\text{ for any section }\vec F\text{ of }\mathcal{H}.$$

\begin{definition}[parametrized bundles]Let $\hat E\subset\mathbb{R}^S$ be
any set and let $E\subset\mathbb{R}^n$ be a compact set. We call a point
$\xi\in\hat E$ a parameter. For each $\xi\in\hat E$ let $\mathcal{H}^\xi=\{H^\xi_{\vec
x}\}_{\vec x\in E}$ be bundle over $E$. The family  $\mathcal{H}=\{H^\xi_{\vec
x}\}^{\xi\in\hat E}_{\vec x\in
E}$  is called a \textit{parametrized bundle}. The \textit{stable Glaeser
refinement} of $\mathcal{H}$ is the parametrized bundle $\tilde{\mathcal{H}}=\{\tilde
H^\xi_{\vec x}\}^{\xi\in\hat
E}_{\vec x\in
E}$, where for each fixed $\xi\in\hat E$, $\{\tilde
H^\xi_{\vec x}\}_{\vec x\in
E}$ is the stable Glaeser refinement of $\mathcal{H}^\xi$.

A parametrized
bundle $\mathcal{H}=\{H^\xi_{\vec x}\}^{\xi\in\hat
E}_{\vec x\in
E}$ is called \textit{semi-algebraic} if the set $$\{(\xi,\vec
x,\vec P)\in\mathbb{R}^S\times\mathbb{R}^n\times\mathcal{P}^\mathcal{D}:\xi\in\hat
E, \vec x\in E\text{ and
}\vec P\in H^\xi_{\vec {x}}\}\subset \mathbb{R}^S\times \mathbb{R}^n\times\mathcal{P}^\mathcal{D}$$
is semi-algebraic. When we say that we are \textit{given} (respectively \textit{compute})
a parametrized bundle, we mean that we are given (resp. compute) this set.\end{definition}

The following theorem is a generalization of (\ref{BB_label_5}), and is (again)
a routine application
of standard properties of semi-algebraic sets. 

\begin{theorem}\label{thm_on_glaeser_ref_of_parameter_bund}Let $\mathcal{H}$
 be a semi-algebraic parametrized bundle, and let  $\tilde{\mathcal{H}}$
be its stable Glaeser refinement. Then, $\tilde{\mathcal{H}}$ is semi-algebraic
as well, and moreover  $\tilde{\mathcal{H}}$ can be computed if $\mathcal{H}$
 is given. \end{theorem}

\subsection{The algorithm}

For the remainder of this section we fix the following: let $I\lhd\mathcal{P}_0^{m}(\mathbb{R}^n)$
be an ideal, let $Q_1,\dots, Q_L$ be a basis of $I$ (as a vector space),
denote $\Omega:=Allow(I)$ and let $p\in\mathcal{P}_0^{m}(\mathbb{R}^n)$ be some
polynomial. We further assume that $\Omega\neq\emptyset$, as if $\Omega=\emptyset$
then we already know by Corollary \ref{cor_no_allowed_directions} that any
jet in $\mathcal{P}_0^{m}(\mathbb{R}^n)$ is implied by $I$, and so $cl(I)=\mathcal{P}_0^{m}(\mathbb{R}^n)$.

\

\begin{definition}\label{def_of_condition_0}We say that condition $\mathcal{C}(A,\epsilon,\delta,r,\rho;p,Q_1,\dots,Q_L)$
holds if there exist functions $F,S_1,S_2,\dots,S_L\in C^m(\text{Ann}_4(\rho))$
satisfying (\ref{new_def_of_implied_label_1}), (\ref{new_def_of_implied_label_2})
and (\ref{new_def_of_implied_label_3}).\end{definition} Definition \ref{def_of_condition_0}
immediately implies the following lemma:

\begin{lemma}\label{lemma_on_conditions_number_1}$I$ implies $p$ if and only
if there exists $A>0$ such that for any $\epsilon>0$
there exist $\delta,r>0$ for which, for all $\rho\in(0,r]$, condition $\mathcal{C}(A,\epsilon,\delta,r,\rho;p,Q_1,\dots,Q_L)$
holds.\end{lemma}

Scaling $F,S_1,S_2,\dots,S_L\in C^m(\text{Ann}_4(\rho))$
that satisfy $\mathcal{C}(A,\epsilon,\delta,r,\rho;p,Q_1,\dots,Q_L)$ by setting
$$\tilde F(\vec x):=\epsilon^{-1}\rho^{-m}F(\rho\vec x)\text{ and }\tilde
S_{l}(\vec x):=A^{-1}S_{l}(\rho\vec x)\text{ for any }1\leq l\leq L,$$ one
easily sees that condition    $\mathcal{C}(A,\epsilon,\delta,r,\rho;p,Q_1,\dots,Q_L)$
is equivalent to the following condition  $\mathcal{C}^{*}(A,\epsilon,\delta,r,\rho;p,Q_1,\dots,Q_L)$:

\begin{definition}We say that condition $\mathcal{C}^{*}(A,\epsilon,\delta,r,\rho;p,Q_1,\dots,Q_L)$
holds if there exist functions $\tilde F,\tilde S_1,\tilde S_2,\dots,\tilde
S_L\in C^m(\{\frac{1}{4}<\abs{\vec x}<4\})$ such that the following hold:

\begin{gather}\label{algoritm_section_label_1}\abs{\partial^\alpha
\tilde F(\vec x)}\leq 1 \text { for }\frac{1}{4}<\abs{\vec
x}<4
\text{ and all }\abs{\alpha}\leq m; \end{gather}
\begin{gather}\label{algoritm_section_label_2}\abs{\partial^\alpha \tilde
S_l(\vec
x)}\leq 1 \text { for }\frac{1}{4}<\abs{\vec x}<4
\text{, all }\abs{\alpha}\leq m \text{ and all }1\leq l\leq L; \end{gather}
\begin{multline}\label{algoritm_section_label_3}p(\rho \vec x)=\epsilon \rho^m
\tilde F(\vec x)+A\tilde S_1(\vec x)Q_1(\rho \vec  x)+\dots+A\tilde S_L(\vec
x)Q_L(\rho \vec  x) \\ \text {for all }\vec x\text{ such that }\frac{1}{2}<\abs{\vec x}<2
\text{ and }\text{dist}(\frac{\vec x}{\abs{\vec x}},\Omega)<\delta.
\end{multline}\end{definition}

So we in fact proved the following lemma (that follows from Lemma \ref{lemma_on_conditions_number_1}):

\begin{lemma}\label{lemma_on_conditions_number_2}$I$ implies $p$ if and only
if there exists $A>0$ such that for any $\epsilon>0$
there exist $\delta,r>0$
for which, for all $\rho\in(0,r]$, condition $\mathcal{C^{*}}(A,\epsilon,\delta,r,\rho;p,Q_1,\dots,Q_L)$
holds.\end{lemma}

We will now introduce another condition, and then see how it relates to condition
$\mathcal{C}^*(A,\epsilon,\delta,r,\rho;p,Q_1,\dots,Q_L)$.

\begin{definition}We say that condition $\mathcal{C}^{**}(\epsilon,\delta,r,\rho;p,Q_1,\dots,Q_L;A)$
holds if there exist functions $F^*,S^*_1,S^*_2,\dots,S^*_L\in C^m(\mathbb{R}^n)$
such that the following hold:

\begin{gather}\label{algoritm_section_label_4}\abs{\partial^\alpha
F^*(\vec x)}\leq A  \text { for all }\vec x\in\mathbb{R}^n
\text{ and all }\abs{\alpha}\leq m; \end{gather}
\begin{gather}\label{algoritm_section_label_5}\abs{\partial^\alpha S^*_l(\vec
x)}\leq A \text { for all }\vec x\in\mathbb{R}^n
\text{, all }\abs{\alpha}\leq m \text{ and all }1\leq l\leq L; \end{gather}
\begin{multline}\label{algoritm_section_label_6}p(\rho \vec x)=\epsilon \rho^m
F^*(\vec x)+S^*_1(\vec x)Q_1(\rho
\vec  x)+\dots+S^*_L(\vec
x)Q_L(\rho \vec  x) \\ \text {for all }\vec x\text{ such that }\frac{1}{2}<\abs{\vec x}<2
\text{ and }\text{dist}(\frac{\vec x}{\abs{\vec x}},\Omega)<\delta.
\end{multline}\end{definition}

Fix $\chi\in C^\infty(\mathbb{R}^n)$ such that $\chi=1$ on $\{\frac{1}{2}<\abs{\vec
x}<2\}$ and   $\chi=0$  outside $\{\frac{1}{4}<\abs{\vec
x}<4\}$, and denote $C_{\chi}=\abs{\abs{\chi}}_{C^m}$ (this is a constant
depending only $m$ and $n$). Assume that condition $\mathcal{C}^*(A,\epsilon,\delta,r,\rho;p,Q_1,\dots,Q_L)$
holds. Then, there exists a constant $\hat C>0$ depending only
on $C_\chi$ and $m$ (so only on $m$ and $n$) such that condition $\mathcal{C}^{**}(\epsilon,\delta,r,\rho;p,Q_1,\dots,Q_L;\hat
CA+\hat C)$ holds: indeed, replacing the functions $\tilde F,\tilde S_1,\tilde
S_2,\dots,\tilde
S_L\in C^m(\{\frac{1}{4}<\abs{\vec x}<4\})$
that satisfy $\mathcal{C^{*}}(A,\epsilon,\delta,r,\rho;p,Q_1,\dots,Q_L)$
by setting
$$F^*(\vec x):=\chi(\vec x)\cdot\tilde F(\vec x)\text{ and }S^{*}_l(\vec
x):=A\cdot \chi(\vec x)\cdot \tilde
S_{l}(\vec x)\text{ for any }1\leq l\leq L,$$ one
easily sees that condition  $\mathcal{C}^{**}(\epsilon,\delta,r,\rho;p,Q_1,\dots,Q_L;\hat
CA+\hat C)$ holds.
 
Vise versa, now assume that condition $\mathcal{C}^{**}(\epsilon,\delta,r,\rho;p,Q_1,\dots,Q_L;A)$
holds. Then,  condition $\mathcal{C}^{*}(A,A\epsilon,\delta,r,\rho;p,Q_1,\dots,Q_L)$
holds: indeed, replacing the functions $F^*,S^*_1,S^*_2,\dots,S^*_L\in C^m(\mathbb{R}^n)$
that satisfy $\mathcal{C^{**}}(A,\epsilon,\delta,r,\rho;p,Q_1,\dots,Q_L)$
by setting
$$\tilde F(\vec x):=A^{-1}\cdot F^*(\vec x)\text{ and }\tilde S_l:=A^{-1}\cdot
\tilde S_l(\vec
x)\text{ for any }1\leq l\leq L,$$ one
easily sees that condition  $\mathcal{C}^{*}(A,A\epsilon,\delta,r,\rho;p,Q_1,\dots,Q_L)$
holds.

\

We found that there exists a constant $\hat C$ depending only on $m,n$ such
that $$\mathcal{C}^*(A,\epsilon,\delta,r,\rho;p,Q_1,\dots,Q_L) \implies \mathcal{C}^{**}(\epsilon,\delta,r,\rho;p,Q_1,\dots,Q_L;\hat
CA+\hat C)$$ and $$\mathcal{C}^{**}(\epsilon,\delta,r,\rho;p,Q_1,\dots,Q_L;A)
\implies \mathcal{C}^{*}(A,A\epsilon,\delta,r,\rho;p,Q_1,\dots,Q_L).$$

\

We conclude that the following lemma holds (and now follows easily from Lemma
\ref{lemma_on_conditions_number_2}):

\begin{lemma}\label{lemma_on_conditions_number_3} $I$ implies $p$ if and
only if there exists $A>0$ such that for any $\epsilon>0$
there exist $\delta,r>0$
for which, for all $\rho\in(0,r]$, condition $\mathcal{C}^{**}(\epsilon,\delta,r,\rho;p,Q_1,\dots,Q_L;A)$
holds.\end{lemma}

\

We are now ready to introduce the relevant parametrized bundle. We define
a bundle \begin{gather}\label{algoritm_section_label_7}\mathcal{H}=\{H^\xi_{\vec
x}\}^{\xi\in\hat
E}_{\vec x\in
E(\xi)}\end{gather} where \begin{gather}\label{algoritm_section_label_8}\hat
E:=(\epsilon,\delta,r,\rho,p,Q_1,\dots,Q_L)\in (0,\infty)^4\times \mathcal{P}^{L+1};\end{gather}\begin{gather}\label{algoritm_section_label_9}E:=\{\vec
x\in\mathbb{R}^n:\frac{1}{2}\leq \abs{\vec x}\leq2\},\end{gather}and for
each $\xi\in\hat E$ and $\vec x\in E$ we define \begin{multline}\label{algoritm_section_label_10}H_{\vec
x}^\xi:=\{(P_0,P_1,\dots,P_L)\in\mathcal{P}^{L+1}:p(\rho\vec x)=\epsilon\rho^m
P_0(\vec x)+\sum\limits_{l=1}^{L}P_l(\vec x)Q_l(\rho
\vec  x)\} \\ \text{if }\text{dist}(\frac{\vec x}{\abs{\vec x}},\Omega)<\delta\text{ and }\frac{1}{2}<\abs{\vec x}<2;\end{multline}
\begin{gather}\label{algoritm_section_label_11}H_{\vec
x}^\xi:=\mathcal{P}^{L+1}\text{ otherwise}.\end{gather}

\begin{remark}\label{remark_our_paramet_bund_is_semi_algeb}The parametrized
bundle ${\mathcal{H}}$ given by (\ref{algoritm_section_label_7})-(\ref{algoritm_section_label_11})
is semi-algebraic. That follows from the fact that this bundle is given by
formulae that are first order
definable in the language of real fields with parameters from $\mathbb{R}$.
In order to see that the set $\Omega:=Allow(I)$ is semi-algebraic, note
that, by the triangle inequality  we have the following formula for this set:
\begin{multline*}Allow(I)=\{\omega\in
S^{n-1}:\text{ there do not exist } c,\delta,r>0\text{ such that }\\ \frac{\abs{Q_1(\vec
x)}+\abs{Q_2(\vec x)}+\dots+\abs{Q_L(\vec x)}}{\abs{\vec
{x}}^m}>c\text{ for any }\vec x\in\Gamma(\omega,\delta,r)\}.\end{multline*}We
leave it to the reader to verify the remaining details.\end{remark}

Comparing (\ref{algoritm_section_label_7})-(\ref{algoritm_section_label_11})
with (\ref{algoritm_section_label_4})-(\ref{algoritm_section_label_6}) we
see that for a fixed  $\xi=(\epsilon,\delta,r,\rho;p,Q_1,\dots,Q_L)$, condition
$\mathcal{C}^{**}(\epsilon,\delta,r,\rho;p,Q_1,\dots,Q_L;A)$ holds if and
only if $\mathcal{H}^\xi$ admits a section with $C^m$-norm at most $A$. 

\

We conclude that the following lemma holds (now follows easily from Lemma
\ref{lemma_on_conditions_number_3})

\begin{lemma}\label{lemma_on_conditions_number_4}$I$ implies $p$ if and only
if there exists $A>0$ such that
for any $\epsilon>0$
there exist $\delta,r>0$
for which, for all $\rho\in(0,r]$, the bundle $\mathcal{H}^\xi$, with $\xi=(\epsilon,\delta,r,\rho;p,Q_1,\dots,Q_L)$,
admits a section with $C^m$-norm at most $A$.\end{lemma}

\

Now let $\tilde{\mathcal{H}}=\{\tilde H^\xi_{\vec x}\}^{\xi\in\hat
E}_{\vec x\in
E}$ be the stable Glaeser refinement of $\mathcal{H}$, and let $k^\#$ be
the integer constant from Theorem \ref{theorem_on_bundles}. Then Lemma \ref{lemma_on_conditions_number_4}
and Theorem \ref{theorem_on_bundles} imply the following:

\begin{proposition}\label{lemma_on_conditions_number_5}

$I$ implies $p$ if and only if there exists $M>0$ such that
for any $\epsilon>0$
there exist $\delta,r>0$
for which, for all $\rho\in(0,r]$, denoting $\xi=(\epsilon,\delta,r,\rho;p,Q_1,\dots,Q_L)$,
the following holds:\begin{multline}\label{algorithn_section_label_12}\text{For
any }\vec x_1,\dots,\vec x_{k^\#}\in E \text{ there exist }P_1\in \tilde
H^\xi_{\vec x_1},\dots,P_k\in
\tilde H^\xi_{\vec x_{k^\#}}\text{ such that} \\ \abs{\partial^\alpha P_i(\vec
x_i)}\leq M\text{ for all }1\leq i\leq
k^\# \text{ and all }\abs{\alpha}\leq m\text{ }\text{ }\text{ }\text{ }\text{
}\text{ }\text{ }\text{ }\text{ }\text{ and} \\ \abs{\partial^\alpha(P_i-P_j)(\vec
x_i)}\leq
M\abs{\vec x_i-\vec x_j}^{m-\abs{\alpha}}\text{ for all }1\leq i, j\leq k^\# \text{ }(\vec x_i,\vec x_j\text{ distinct}) \text{
and all }\abs{\alpha}\leq m. \end{multline}\end{proposition}

\

Proposition \ref{lemma_on_conditions_number_5}  gives rise to an algorithm to compute the closure of a given ideal $I$ of $m$-jets: Given a basis $Q_1,\dots, Q_L$ for $I$ as a vector space, we form the parametrized bundle $\mathcal{H}$ defined by (\ref{algoritm_section_label_7})-(\ref{algoritm_section_label_11}), then compute its stable Glaeser refinement $\tilde{\mathcal{H}}$. We know that $\tilde{\mathcal{H}}$ is semi-algebraic thanks to Remark \ref{remark_our_paramet_bund_is_semi_algeb} and (\ref{BB_label_5}). Once $\tilde{\mathcal{H}}$ is known, Proposition \ref{lemma_on_conditions_number_5}  expresses the condition that $I$ implies $p$ as a first order definable condition in the language of real fields with parameters from
$\mathbb{R}$. Consequently, standard semi-algebraic technology allows us to compute the set of all $p$ implied by $I$ as a semi-algebraic set. Since the set of such $p$ is a vector subspace of $\mathcal{P}_0^{m}(\mathbb{R}^n)$ another application of standard semi-algebraic technology computes a basis for that subspace. Thus, in principle, we have computed the closure of a given ideal $I$ of $m$-jets.

We have proved the following result.

\begin{theorem}\label{thm_closure_is_computable}Let $I\lhd\mathcal{P}_0^m(\mathbb{R}^n)$
be an ideal. Then, the above procedure computes (in principle) the closure of $I$. \end{theorem}

\begin{remark}\label{remark_on_semi_algebaric_map_exists}Fix $L\in\mathbb{N}$.
Then, Theorem \ref{thm_closure_is_computable} together with standard semi-algebraic
geometry arguments show that there exists a semi-algebraic map that maps
any basis of an ideal $Q_1,\dots,Q_L\in\mathcal{P}_0^m(\mathbb{R}^n)$ to
a basis for the closure of the ideal, i.e.,  a basis for $cl(\langle Q_1,\dots,Q_L\rangle_m)$.
\end{remark}

Recall that $\mathcal{P}_0^{m}(\mathbb{R}^n)$ is a finite
dimensional vector space and fix $P_1,\dots,P_{\mathcal{D}^*}$, some basis
(as a vector space) of $\mathcal{P}_0^{m}(\mathbb{R}^n)$. Then, we can naturally
identify $\mathcal{P}_0^{m}(\mathbb{R}^n)$ with $\mathbb{R}^{\mathcal{D}^*}$.
Fix an integer $1\leq L\leq \mathcal{D}^*$. We can now identify each ordered
set of jets $Q_1,\dots Q_{L}\in\mathcal{P}_0^{m}(\mathbb{R}^n)$, not necessarily
all different jets, as a point in $\mathbb{R}^{L\cdot \mathcal{D}^*}$. Then,
each point in $\mathbb{R}^{L\cdot \mathcal{D}^*}$ defines a vector space
$I=\text{span}_{\mathbb{R}}\{
Q_1,\dots Q_{L}\}$ of dimension at most $L$. So we can think of any point
in  $\mathbb{R}^{L\cdot \mathcal{D}^*}$ as a vector space in $\mathcal{P}_0^{m}(\mathbb{R}^n)$
of dimension at most $L$, where not every two different points represent
necessarily different vector spaces.

Note that the condition "$Q_1,\dots,Q_L$ are linearly independent" is first
order
definable in the language of real fields with parameters from $\mathbb{R}$.
Moreover, note that the condition "$\text{span}_{\mathbb{R}}\{
Q_1,\dots Q_{L}\}=\langle
Q_1,\dots Q_{L} \rangle_m$" 
is also first
order
definable in the language of real fields with parameters from $\mathbb{R}$.
We conclude that the set of points in $\mathbb{R}^{L\cdot \mathcal{D}^*}$
that represent bases of ideals of dimension $L$ is a semi-algebraic set.

Now Remark \ref{remark_on_semi_algebaric_map_exists} implies the following:

\begin{theorem}Fix $1\leq L\leq \dim\mathcal{P}_0^m(\mathbb{R}^n)$. Then,
the set of points in $\mathbb{R}^{L\cdot \mathcal{D}^*}$
that represent bases of closed ideals of dimension $L$ is semi-algebraic.
\end{theorem}

We do not know whether the set of points in $\mathbb{R}^{L\cdot \mathcal{D}^*}$
that represent bases of ideals of dimension $L$ of the form $I^m(E)$ for
some closed $E\subset\mathbb{R}^n$ is semi-algebraic. We also do not know
whether the set of points in $\mathbb{R}^{L\cdot \mathcal{D}^*}$
that represent bases of ideals of dimension $L$ of the form $I^m(E)$ for
some semi-algebraic closed $E\subset\mathbb{R}^n$ is semi-algebraic. These
question are closely related to Question \ref{main_question} and Question
\ref{semi_algebraic_question}.


\begin{thebibliography}{MMM}

\bibitem[BaCR]{BaCR} S.~Basu, R.~Pollak and M-F.~Roy, {\em Algorithms in real
geometry}, Volume 10 of \textit{Algorithms and computations in Mathematics}, Springer-Verlag, Berlin, Second Edition, 2016. 

\bibitem[BoCR]{BCR} J.~Bochnak, M.~Coste and M-F.~Roy, {\em Real algebraic geometry}, Ergebnisse der
Mathematik
und ihrer Grenzgebiete, 36, Springer-Verlag, Berlin Heidelberg, 1998. http://dx.doi.org/10.1007/978-3-662-03718-8. 
\bibitem[vdD]{vdD} L.~van den Dries, {\em Tame Topology and O-minimal Structures},
London Mathematical Society Lecture Notes Series, vol. 248, Cambridge University Press, Cambridge, 1998, ISBN: 0 521 59838 9 (paperback).


\bibitem[F1]{F1} C.~Fefferman, {\em Whitney's extension problem for $C^m$}, Ann. of Math. (2) 164 (2006), no.1, 313--359, DOI:10.4007/annals.2006.164.313.

\bibitem[F2]{F} C.~Fefferman, {\em A few unsolved problems (the Ninth Whitney Problems Workshop)}, 2016. https://cms-math.net.technion.ac.il/files/2016/06/Charles-Fefferman-A-few-unsolved-problems-June-30-2016.pdf 
\bibitem[FI]{FI} C.~Fefferman, A.~Israel, {\em Fitting smooth functions to data}, CBMS Regional Conference Series in Mathematics, American Mathematical Society, 2020.

\bibitem[FL1]{FL1} C.~Fefferman, G.~K.~Luli, {\em The Brenner-Hochster-Koll\'ar and Whitney problems for vector-valued functions and jets}, Rev. Mat. Iberoam. (2013) 30 (3) pp. 875--872, DOI:10.4171/rmi/801.

\bibitem[FL2]{FL2} C.~Fefferman, G.~K.~Luli, {\em Solutions to a system of equations for $C^m$ functions}, Rev. Mat. Iberoam.
(2021) 37 (3) pp. 911--963, DOI:10.4171/rmi/1217.

\bibitem[FL3]{FL3} C.~Fefferman, G.~K.~Luli, {\em Generators for the $C^m$-closures of ideals}, Rev. Mat. Iberoam.
(2021) 37 (3) pp. 965--1006, DOI:10.4171/rmi/1218.

\bibitem[FS]{FS2} C.~Fefferman, A.~Shaviv, {\em Classification of implication closed ideals in certain
rings of jets}, 2022, preprint.

%\bibitem[P]{P} R.~S.~Palais, {\em The Morse lemma for Banach spaces}, Bull. %Amer. Math. Soc.
%75 (1969) pp. 968--971, DOI:10.1090/S0002-9904-1969-12318-9.

\bibitem[W]{W} H.~Whitney, {\em Analytic extensions of differentiable functions defined in closed sets}, Trans. Amer. Math. Soc. 36 (1934) pp. 63--89, DOI:10.1090/S0002-9947-1934-1501735-3.



\end{thebibliography}
\end{document}